\documentclass[preprint,12pt]{elsarticle}

\usepackage{epsfig}
\usepackage{epstopdf}
\usepackage{databib}
\usepackage{amssymb}
\usepackage{amsmath}
\usepackage{amsthm}
\usepackage{amsfonts}
\usepackage{CJK}
\usepackage{pifont}
\usepackage{natbib}
\usepackage{geometry}
\usepackage{fleqn}
\usepackage{graphicx}
\usepackage{txfonts}
\usepackage[colorlinks,linkcolor=red,anchorcolor=blue,citecolor=blue]{hyperref}
\usepackage{lineno,mathrsfs,verbatim,bm,enumerate,caption,float,CJK}

\newtheorem{theorem}{Theorem}[section]
\newtheorem{remark}{Remark}[section]
\newtheorem{lemma}{Lemma}[section]
\newtheorem{example}{Example}[section]
\newtheorem{corollary}{Corollary}[section]
\begin{document}

\begin{frontmatter}

\title{Optimal asymptotic analyses on Laguerre and Hermite orthogonal approximation for functions of algebraic and logarithmic regularities\tnoteref{mytitlenote}}
\tnotetext[mytitlenote]{This work was supported by the National Natural Science Foundation of China (No. 12271528), the Natural Science Foundation of Jiangsu Province (BK20241087).}

\author[mymainaddress1,mymainaddress2]{Yali Zhang}
\ead{ylzhang\_math@163.com}

\author[mymainaddress3]{Guidong Liu}
\ead{liugd@nau.edu.cn}

\author[mymainaddress2]{Shuhuang Xiang\corref{mycorrespondingauthor}}
\cortext[mycorrespondingauthor]{Corresponding author}
\ead{xiangsh@csu.edu.cn}

\address[mymainaddress1]{School of Mathematics and Statistics, Yancheng Teachers University, Yancheng 224000,  P. R. China}

\address[mymainaddress2]{School of Mathematics and Statistics, Central South University, Changsha 410083,  P. R. China}

\address[mymainaddress3]{School of Mathematics, Nanjing Audit University, Nanjing 211815,  P. R. China}

\begin{abstract}
Based on the Hilb-type formula and van der Corput-type lemmas, we present optimal asymptotic estimates for the decay of the Laguerre and Hermite coefficients for functions with algebraic and logarithmic singularities, which in turn yield the convergence rates of the corresponding spectral orthogonal projections. Numerous examples are provided to verify the optimality of these asymptotic results.
\end{abstract}


\end{frontmatter}


\section{Introduction}
The approximation of functions {using} orthogonal polynomials plays a fundamental role in numerical analysis and scientific computing, particularly in spectral methods for solving differential equations, Gaussian quadrature rules, and {expansions in special functions} \cite{Szego,Shen,chen2023log,Guo,Xiang,Xiang2012,gil2018asymptotic}. Among the various families of orthogonal polynomials, Laguerre and Hermite polynomials are of particularly significance for representing functions over semi-infinite and infinite domains, respectively. Their effectiveness in function approximation has led to applications in various fields, including quantum mechanics, signal processing and numerical solutions of partial differential equations \cite{fatyanov2011high,terekhov2013fast,mikhailenko1999spectral,mastryukov2008numerical,
colton1984analytic,Javier,funaro1991approximation,schumer1998vlasov}. {Research on these polynomials has yielded many new and important results \cite{agarwal2017certain,prajapati2014some,avazzadeh2023optimal,hassani2023study,avazzadeh2023optimization}}.

A prominent advantage of orthogonal polynomial approximations is their well-defined orthogonality properties, which facilitate stable numerical computations and rapid convergence. The rate of convergence is primarily determined by the regularity of the function being approximated. Classical results in spectral approximation theory focus on analytic or highly smooth functions, for which the expansion coefficients decay rapidly, {leading to exponential convergence in appropriate norms}. For example, Elliott and Tuan \cite{elliott1974asymptotic} established contour integral representations for the Laguerre and Hermite coefficients of analytic functions. Wang in \cite{wang2024convergence}, and Wang and Zhang in \cite{wang2023convergence} investigated the root-exponential convergence of Laguerre and Hermite approximations for such functions. For functions with lower regularity, Shen and Wang \cite{shen2009some} and Xiang \cite{Xiang2012} analyzed the algebraic convergence rates {of order} $\mathcal{O}(n^{-s})$, where $s$ is a parameter that depends on the regularity of the underlying functions. However, these results become suboptimal when the regularity parameter $s$ is non-integer. More recently, Zhang et al. \cite{Zhang,Zhang2024} investigated optimal pointwise estimates for Laguerre approximations by a Hilb-type asymptotic formula for Laguerre polynomials.

In this paper, we {study} the asymptotic behavior of Laguerre and Hermite expansion coefficients for functions with endpoint or interior singularities. We {first} derive several asymptotic estimates for integrals involving highly oscillatory Bessel functions. Following Hilb-type formulas, these estimates are subsequently applied to integrals involving Laguerre and Hermite polynomials. By {employing} Rodrigues' formulas, we then establish the decay rates of the Laguerre and Hermite coefficients for functions with endpoint or interior singularities. It is important to emphasize that all of the derived estimates are optimal in the sense that the convergence rates cannot be improved, which will be {confirmed} by ample numerical experiments. Finally, these results are used to characterize the decay of the corresponding Laguerre and Hermite spectral orthogonal projections.

The structure of this paper is as follows. In Section 2, we review the fundamental properties of Laguerre polynomials and Hermite polynomials. In Section 3, we derive the generalized van der Corput-type Lemmas for Bessel transforms. Sections 4 and 5 present the optimal decay rates of the Laguerre and Hermite coefficients, respectively, as well as the convergence rates of the corresponding spectral orthogonal projections for functions with algebraic and logarithmic singularities. Finally, Section 6 provides a brief summary of the main results and conclusions.

\section{Preliminaries}
\subsection{Gneralized Laguerre polynomials}
The generalized Laguerre polynomials $L_{n}^{(\alpha)}(x)$ form an orthogonal basis on the semi-infinite interval $[0,+\infty)$ with respect to the weight function $\omega_{\alpha}(x)=x^{\alpha}{\rm e}^{-x}$, where $\alpha>-1$. Specifically, they satisfy the orthogonality property
\begin{align}\label{eq1}
  \int_{0}^{+\infty}\omega_{\alpha}(x)L_{n}^{(\alpha)}(x)L_{m}^{(\alpha)}(x)\,\mathrm{d}x=\sigma_{n}^{(\alpha)}
  \delta_{m,n},\quad \sigma_n^{(\alpha)}=\frac{\Gamma(n+\alpha+1)}{n!},
\end{align}
where $\delta_{m,n}$ denotes the Kronecker delta symbol.

Let $f(x)$ be a suitably smooth function on $[0,\infty)$.
Then it can be expanded in a Laguerre series \cite[p. 265]{Shen} as
\begin{align*}
   f(x) =\sum_{n=0}^{\infty} a_n (\alpha) L_{n}^{(\alpha)}(x),
\end{align*}
where the coefficients $a_n(\alpha)$ are given by
\begin{align}\label{eq3}
 a_n (\alpha)=\frac{1}{\sigma_n^{(\alpha)}} \int_0^{\infty}  f(x) L_{n}^{(\alpha)}(x)\omega_{\alpha}(x) \mathrm{d}x.
\end{align}
{A practical approximation of} $f(x)$ is obtained by truncating the Laguerre series after the first $N+1$-terms
\begin{align*}
   S_N^{(\alpha)}[f](x) =\sum_{n=0}^{N} a_n (\alpha) L_{n}^{(\alpha)}(x).
\end{align*}
The {corresponding} approximation error in the weighted $L^2$-norm {with respect to} $\omega_{\alpha}(x)$ is given by
\begin{align}\label{eq4}
\|f(x)-S_N^{(\alpha)}[f](x)\|_{L_{\omega_{\alpha}}^2[0,\infty)}
=\sqrt{\sum_{n=N+1}^{\infty} a_n^2 (\alpha) \sigma_{n}^{(\alpha)}}.
\end{align}

\subsection{Hermite orthogonal polynomials}
The Hermite polynomials $H_{n}(x)$ form an orthogonal basis on the entire real line $(-\infty,+\infty)$ with respect to the Gaussian weight function $\omega(x)=\mathrm{e}^{-x^{2}}$. They are widely used in probability theory, quantum mechanics and signal processing. For a function $g(x)$ defined on $(-\infty,+\infty)$, the Hermite series expansion takes the form \cite[p. 270]{Shen}
\begin{align*}
{g(x)} =\sum_{n=0}^{\infty} h_n H_{n}(x),
\end{align*}
where the Hermite coefficients $h_{n}$ are defined as
\begin{align}\label{eq6}
   h_n=\frac{1}{\gamma_n} \int_{-\infty}^{+\infty} g(x) H_{n}(x) \omega(x)\mathrm{d}x,
   \quad \gamma_n=\sqrt{\pi} 2^n n!.
\end{align}

The $N$th partial sum of the Hermite series provides an approximation to $g(x)$
\begin{align*}
  S_N[g](x) =\sum_{n=0}^{N} h_n H_{n}(x).
\end{align*}
{with the corresponding approximation error in the weighted $L^2$-norm with respect to $\omega(x)$ given by}
\begin{align}\label{eq7}
\|g(x)-S_N[g](x)\|_{L_{\omega}^2(-\infty,\infty)}
=\sqrt{\sum_{n=N+1}^{\infty} h_n^2  \gamma_{n}}.
\end{align}

Moreover, Hermite polynomials are closely connected to generalized Laguerre polynomials. In particular, they satisfy the following identities (see \cite{Szego,Shen})
\begin{align}\label{eq8}
\aligned
&H_{2n}(x)=(-1)^n 2^{2n} n! L_n^{(-1/2)}(x^2),
\\&H_{2n+1}(x)=(-1)^n 2^{2n+1} n! x L_n^{(1/2)}(x^2),
\endaligned
\end{align}
where $L_n^{(-1/2)}(x^2)$ and $L_n^{(1/2)}(x^2)$ are generalized Laguerre polynomials of degree $n$ {with parameters} $\alpha=-1/2$ and $1/2$, respectively. These relations illustrate the structural interplay between Hermite and Laguerre systems, which is of particular interest in asymptotic and approximation analyses.

\subsection{The asymptotic properties of Laguerre and Hermite polynomials}
From an asymptotic point of view as $n\to\infty$, the generalized Laguerre {polynomials exhibit} a close {connection with the} Bessel function of the first kind. {This relationship is formalized by the following Hilb-type asymptotic formula}.
\begin{lemma}{\rm (Asymptotic formula of Hilb's type \cite[Theorem 8.22.4]{Szego})} \label{lemma2.1}
Let $\alpha>-1$ and define $\tilde{n}=n+(\alpha+1)/2$. As $n\to\infty$, the generalized Laguerre polynomial satisfies
\begin{align}\label{eq9}
\aligned
{\rm e}^{-x/2} x^{\alpha/2} L_{n}^{(\alpha)}(x)&=\frac{\tilde{n}^{-\alpha/2}\Gamma(n+\alpha+1)}{n!}
J_{\alpha}\big(2\sqrt{\tilde{n}x}\big)
+\left\{
\aligned
&  x^{5/4}\mathcal{O}\Big(n^{\frac{\alpha}{2}-\frac{3}{4}}\Big), & c n^{-1} \leq x \leq \omega,
\\& x^{\alpha/2+2}\mathcal{O}(n^{\alpha}), &0< x \leq c n^{-1},
\endaligned \right.
\endaligned
\end{align}
where $c$ and $\omega$ are fixed positive constants, and $J_{\alpha}(z)$ denotes the Bessel function of the first kind of order $\alpha$. The bounds above hold uniformly for $0<x\leq \omega$. In the special case $\alpha=0$, the last bound in \eqref{eq9} is to be replaced by $x^2 \log\big(x^{-1}n^{-1}\big)$.
\end{lemma}

In addition to their asymptotic representations, the maximum behavior of Laguerre and Hermite polynomials is {of significant interest} in approximation theory. The following results provide sharp estimates for their maximum magnitudes in weighted forms.

\begin{lemma}{\rm(\cite[Theorem 8.91.2]{Szego})}\label{lemma2.2}
Let $\alpha,\lambda\in\mathbb{R}$, $a>0$ and $0<\eta<4$. As $n\to\infty$, the generalized Laguerre polynomial satisfies
\begin{align}\label{eq10}
\max_{x\geq a} {\rm e}^{-x/2} x^{\lambda} \left|L_{n}^{(\alpha)}(x)\right|
=\left\{\begin{aligned}
&\mathcal{O}\Big(n^{\max\{\lambda-1/2,\alpha/2-1/4\}}\Big),\quad &&\mathrm{if }\, a\leq x\leq(4-\eta)n,\\
&\mathcal{O}\Big(n^{\max\{\lambda-1/3,\alpha/2-1/4\}}\Big),\quad &&\mathrm{if }\, x\geq a.\\
\end{aligned}\right.
\end{align}
\end{lemma}

\begin{lemma}{\rm(\cite[Theorem 8.91.3]{Szego})}\label{lemma2.3}
Let $\lambda\in\mathbb{R}$, $a>0$ and $0<\eta<2$. As $n\to\infty$, the Hermite polynomial satisfies
\begin{align}\label{eq11}
  \max_{|x|\geq a} {\rm e}^{-x^2/2} x^{\lambda} \left|H_{n}(x)\right|
  =\sqrt{2^{n}n!}\left\{\begin{aligned}
  &\mathcal{O}\Big(n^{\max\{\lambda/2-1/4,-1/4\}}\Big),\quad &&\mathrm{if}\,\, a\leq |x|\leq \sqrt{(2-\eta)n},\\
  &\mathcal{O}\Big(n^{\max\{\lambda/2-1/12,-1/4\}}\Big),\quad &&\mathrm{if}\,\, |x|\geq a.
  \end{aligned}\right.
\end{align}
\end{lemma}

These results {play a fundamental role} in establishing precise convergence rates and asymptotic estimates {for} Laguerre and Hermite approximations, {particularly for functions exhibiting singularities}.

\section{Some useful {lemmas}}
In this section, we review several asymptotic estimates for integrals involving highly oscillatory Bessel functions, which arise through the Hilb-type formula \eqref{eq9} in connection with Laguerre polynomials. These results extend the classical van der Corput lemma for Fourier transforms to the following Bessel transforms.

\begin{lemma}[\cite{Xiang2007}]\label{lemma3.1}
Let $b>a>0$, $\psi(x)\in C[a,b]$ and $\psi^{\prime}(x)\in L^1[a,b]$. Then as $\omega\to\infty$, it holds that
\begin{align*}
\int_a^{b}  J_v(\omega x) \psi(x) \mathrm{d}x
=\mathcal{O} \Big(\omega^{-3/2}\Big).
\end{align*}
\end{lemma}

\begin{lemma}[\cite{Xiang}]\label{lemma3.2}
Let $\alpha+v> -1$, $\beta>-1$ and $\mu\in\mathbb{N}$. Suppose that $\psi(x)\in C[0,b]$ and $\psi'(x)\in L^1[0,b]$ for a constant $b>0$. Then, as $\omega\to\infty$, the following asymptotic estimates hold
\begin{align}\label{eq12}
\int_0^{b} \ln^{\mu}(x) x^{\alpha} (b-x)^{\beta}J_v(\omega x)  \psi(x)\mathrm{d}x=
\mathcal{O}\left(\max \left\{ \frac{\ln^{\mu}\omega}{\omega^{\alpha+1}},\frac{1}{\omega^{\min \left\{\beta+3/2,3/2 \right\}}} \right\}\right),
\end{align}
and
\begin{align}\label{eq13}
\int_0^{b} \ln^{\mu}(b-x) x^{\alpha}(b-x)^{\beta}J_v(\omega x)  \psi(x)\mathrm{d}x=
\mathcal{O}\left(\max \left\{ \frac{1}{\omega^{\alpha+1}},\frac{\ln^{\mu}\omega}{\omega^{\min \left\{\beta+3/2,3/2 \right\}}} \right\}\right).
\end{align}
\end{lemma}

{
\begin{lemma}\label{lemma3.3}
Let $\beta>-1$ and $\mu\in { \mathbb{N}}$. For any fixed constants $a$ and $b$ satisfying $0<a<b<\infty$, the following estimates hold uniformly for $t\in [a,b]$
\begin{align}\label{eq14}
\int_a^{t} \ln^{\mu}(x-a) (x-a)^{\beta} J_v(\omega x) \mathrm{d}x=
\mathcal{O}\left(\max \left\{\omega^{-\beta-3/2}\ln^{\mu}\omega, \omega^{-3/2} \right\}\right),
\end{align}
and
\begin{align}\label{eq15}
\int_t^{b} \ln^{\mu}(b-x) (b-x)^{\beta}J_v(\omega x) \mathrm{d}x=
\mathcal{O}\left(\max \left\{\omega^{-\beta-3/2}\ln^{\mu}\omega, \omega^{-3/2} \right\}\right).
\end{align}
as $\omega\to\infty$.
\end{lemma}}
\begin{proof}
Recall that for $z\in[a,a+1/\omega]$, the Bessel function satisfies $J_v(\omega z)=\mathcal{O}\big(\omega^{-1/2}\big)$ (see, e.g., \cite[(1.71.7)]{Szego}). Thus, for $t\in [a,a+1/\omega]$, it follows that
\begin{align*}
\aligned
&\int_a^t \ln^{\mu}(x-a)(x-a)^{\beta} J_v(\omega x) \mathrm{d}x
\\&=\mathcal{O}\big(\omega^{-1/2}\big) \int_a^{t} \ln^{\mu}(x-a)(x-a)^{\beta}  \mathrm{d}x
\\&=\mathcal{O}\big(\omega^{-1/2}\big)\left(\frac{\ln^{\mu}(x-a)  (x-a)^{\beta+1}}{\beta+1}\Big|_{a}^{t}
-\frac{\mu}{\beta+1} \int_a^{t}  \ln^{\mu-1}(x-a)  (x-a)^{\beta} \mathrm{d}x\right)
\\&=\mathcal{O}\big(\omega^{-\beta-3/2}\ln^{\mu}\omega\big)
+\mathcal{O}\big(\omega^{-\beta-3/2}\ln^{\mu-1}\omega\big)+\cdots+\mathcal{O}\big( \omega^{-\beta-3/2}\big)
\\&=\mathcal{O}\big(\omega^{-\beta-3/2}\ln^{\mu}\omega\big).
\endaligned
\end{align*}
For $t\in (a+1/\omega,b]$, we split the integral and derive that
\begin{align*}
\aligned
&\int_a^t \ln^{\mu}(x-a)(x-a)^{\beta} J_v(\omega x) \mathrm{d}x
\\&=\int_a^{a+1/\omega} \ln^{\mu}(x-a)(x-a)^{\beta} J_v(\omega x) \mathrm{d}
+\int_{a+1/\omega}^t \ln^{\mu}(x-a)(x-a)^{\beta} J_v(\omega x) \mathrm{d}x
\\&=\mathcal{O}\big(\omega^{-\beta-3/2}\ln^{\mu}\omega \big)+O\big(\omega^{-3/2}\big)
\\&=\mathcal{O}\left(\max \left\{\omega^{-\beta-3/2}\ln^{\mu}\omega , \omega^{-3/2} \right\}\right),
\endaligned
\end{align*}
where the estimate of the second integral follows from Lemma \ref{lemma3.1}. Then we get Equation \eqref{eq14}. A similar argument applies to the second integral leads to the estimate \eqref{eq15}, by the symmetry of the analysis.
\end{proof}

\begin{lemma}\label{lemma3.4}
Suppose $\beta>-1$, $0<a<b<\infty$, $\mu\in\mathbb{N}$. Suppose that $\psi(x)\in C[a,b]$ and $\psi'(x)\in L^1[a,b]$. Then the following asymptotic results hold as $\omega\to\infty$
\begin{align}\label{eq16}
\quad\int_a^{b} \ln^{\mu}(x-a) (x-a)^{\beta}\psi(x) J_v(\omega x) \mathrm{d}x
=\mathcal{O}\left(\max \left\{\omega^{-\beta-3/2}\ln^{\mu}\omega , \omega^{-3/2} \right\}\right)
\end{align}
and
\begin{align}\label{eq17}
\quad\int_a^{b} \ln^{\mu}(b-x) (b-x)^{\beta}\psi(x) J_v(\omega x) \mathrm{d}x=
\mathcal{O}\left(\max \left\{\omega^{-\beta-3/2}\ln^{\mu}\omega , \omega^{-3/2} \right\}\right).
\end{align}
\end{lemma}

\begin{proof}
According to Equation \eqref{eq14}, we have
\begin{align*}
\aligned
&\int_a^{b} \ln^{\mu}(x-a)  (x-a)^{\beta} \psi(x)J_v(\omega x) \mathrm{d}x
\\&=\int_a^{b} \psi(x)\left[\int_a^x \ln^{\mu}(u-a) (u-a)^{\beta} J_v(\omega u) \mathrm{d}u\right]' \mathrm{d}x
\\&=\psi(b) \int_a^{b} \ln^{\mu}(x-a) (x-a)^{\beta} J_v(\omega x) \mathrm{d}x
-\int_a^{b} \psi'(x)\left[\int_a^x \ln^{\mu}(u-a)(u-a)^{\beta} J_v(\omega u) \mathrm{d}u\right] \mathrm{d}x
\\&\leq C_0 \left( |\psi(b)| +\int_a^{b} |\psi'(x)| \mathrm{d}x \right)
\cdot  \max \left\{\omega^{-\beta-3/2}\ln^{\mu}\omega , {\omega^{-3/2}}\right\}
\\&\leq C  \cdot \max \left\{\omega^{-\beta-3/2}\ln^{\mu}\omega , {\omega^{-3/2}}\right\},
\endaligned
\end{align*}
where $C_0, C$ are some constants independent of $x$.

The proof of \eqref{eq17} follows analogously by applying the same argument to Equation \eqref{eq15}.
\end{proof}

\begin{lemma}\label{lemma3.5}
{Let $a$ and $b$ be fixed finite constants, $\mu\in\mathbb{N}$, $\psi(x)\in C[a,b]$ and $\psi'(x)\in L^1[a,b]$. Then the following asymptotic results hold as $\omega\to\infty$:}

\begin{itemize}
\item Suppose the following conditions are satisfied to ensure the integrability of the integral: (i) If $0 < a < b$, assume $\beta > -1$. (ii) If $0 = a < b$, assume $\alpha + \delta + \beta + v > -1$. (iii) If $a < 0 \leq b$, assume $\alpha + \delta + v > -1$ and $\beta > -1$. (iv) If $a < b < 0$, assume $\beta > -1$. Then we have
      \begin{align}\label{eq18}
      \aligned
      & \int_{a}^{b} x^{\alpha}|x|^{\delta} \ln^{\mu}(x-a)  (x-a)^{\beta} \psi(x) J_v(\omega |x|) \mathrm{d}x
      \\&=\left\{
      \aligned
      &\mathcal{O}\left(\max \left\{\omega^{-\beta-3/2}\ln^{\mu}\omega , \omega^{-3/2} \right\}\right), &0<a<b,
      \\&\mathcal{O}\left(\max \left\{\omega^{-\alpha-\delta-\beta-1}\ln^{\mu}\omega ,\omega^{-3/2} \right\}\right), &0=a<b,
      \\&\mathcal{O}\left(\max \left\{ \frac{1}{\omega^{\alpha+\delta+1}},\frac{\ln^{\mu}\omega}{\omega^{\min \left\{\beta+3/2,3/2\right\}}} \right\}\right), &a<0\leq b,
      \\&\mathcal{O}\left(\max \left\{\omega^{-\beta-3/2}\ln^{\mu}\omega , \omega^{-3/2} \right\}\right), &a<b<0.
      \endaligned\right.
      \endaligned
      \end{align}

\item Suppose the following conditions hold: (i) {If $0<a<b$, assume $\beta>-1$}. (ii) {If $a\leq0<b$, assume $\alpha+\delta+v>-1$ and $\beta>-1$}. (iii) If $a<b=0$, assume $\alpha+\delta+\beta+v>-1$. (iv) If $a<b<0$, assume $\beta>-1$. Then we have
    \begin{align}\label{eq19}
    \aligned
    & \int_{a}^{b} x^{\alpha}|x|^{\delta} \ln^{\mu}(b-x)  (b-x)^{\beta} \psi(x) J_v(\omega |x|) \mathrm{d}x
    \\&=\left\{
    \aligned
    & \mathcal{O}\left(\max \left\{\omega^{-\beta-3/2}\ln^{\mu}\omega, \omega^{-3/2} \right\}\right), &{0<a<b},
    \\&\mathcal{O}\left(\max \left\{ \frac{1}{\omega^{\alpha+\delta+1}},\frac{\ln^{\mu}\omega}{\omega^{\min \left\{\beta+3/2,3/2 \right\}}} \right\}\right), &{a\leq0<b},
    \\&\mathcal{O}\left(\max \left\{\omega^{-\alpha-\delta-\beta-1}\ln^{\mu}\omega ,\omega^{-3/2} \right\}\right), &a<b=0,
    \\& \mathcal{O}\left(\max \left\{\omega^{-\beta-3/2}\ln^{\mu}\omega, \omega^{-3/2} \right\}\right), &a<b<0.
    \endaligned
    \right.
    \endaligned
    \end{align}
\end{itemize}
\end{lemma}

\begin{proof}
The proof is divided into {four} cases:

\noindent \textbf{Case (i): $0<a<b$.} According to Equation \eqref{eq16}, we have
\begin{align*}
\int_a^{b} x^{\alpha+\delta}\ln^{\mu}(x-a) (x-a)^{\beta}\psi(x) J_v(\omega x) \mathrm{d}x
=\mathcal{O} \left(\max \left\{\omega^{-\beta-3/2}\ln^{\mu}\omega, \omega^{-3/2} \right\}\right),
\end{align*}
{where we used the fact that $x^{\alpha+\delta}\psi(x)\in C[a,b]$ and $(x^{\alpha+\delta}\psi(x))^{\prime}\in L^{1}[a,b]$}.

\noindent \textbf{Case (ii): $0=a<b$.} Applying Equation \eqref{eq12}, we obtain
\begin{align*}
\int_0^{b} x^{\alpha+\delta+\beta}\ln^{\mu}(x) \psi(x)J_v(\omega x) \mathrm{d}x=
\mathcal{O} \left(\max \left\{\omega^{-\alpha-\delta-\beta-1}\ln^{\mu}\omega,\omega^{-3/2} \right\}\right).
\end{align*}
{
\noindent \textbf{Case (iii): $a<0\leq b$.} By Equation \eqref{eq13}, we estimate
\begin{equation*}
\aligned
& \int_{a}^b x^{\alpha}|x|^{\delta}  \ln^{\mu}(x-a) (x-a)^{\beta}\psi(x) J_v(\omega |x|) \mathrm{d}x
\\&= (-1)^\delta\int_{a}^0 x^{\alpha+\delta}\ln^{\mu}(x-a) (x-a)^{\beta}\psi(x) J_v(-\omega x) \mathrm{d}x\\
&\quad+\int_{0}^b x^{\alpha+\delta} \ln^{\mu}(x-a) (x-a)^{\beta}\psi(x) J_v(\omega x) \mathrm{d}x
\\&=(-1)^{\alpha}\int_{0}^{-a} x^{\alpha+\delta} \ln^{\mu}(-a-x) (-a-x)^{\beta}\psi(-x) J_v(\omega x) \mathrm{d}x
+\mathcal{O}\left(\max \left\{\omega^{-\alpha-\delta-1},\omega^{-3/2}\right\}\right)
\\&=\mathcal{O}\left(\max \left\{ \frac{1}{\omega^{\alpha+\delta+1}},\frac{\ln^{\mu}(\omega)}{\omega^{\min \left\{\beta+3/2,3/2 \right\}}} \right\}\right)
+\mathcal{O}\left(\max \left\{\omega^{-\alpha-\delta-1},\omega^{-3/2}\right\}\right)
\\&=\mathcal{O}\left(\max \left\{ \frac{1}{\omega^{\alpha+\delta+1}},\frac{\ln^{\mu}\omega}{\omega^{\min \left\{\beta+3/2,3/2 \right\}}} \right\}\right),
\endaligned
\end{equation*}
where we used the fact that $\ln^{\mu}(x-a)(x-a)^{\beta}\psi(x)\in C[0,b]$ and $(\ln^{\mu}(x-a)(x-a)^{\beta}\psi(x))^{\prime}\in L^{1}[0,b]$ for $a<0$ in the second equality. In particular, if $b=0$, the same estimate holds:
\begin{equation*}
\aligned
& \int_{a}^{0} x^{\alpha}|x|^{\delta} \ln^{\mu}(x-a) (x-a)^{\beta}\psi(x) J_v(\omega |x|) \mathrm{d}x
\\&=(-1)^{\alpha} \int_{0}^{-a} x^{\alpha+\delta} \ln^{\mu}(-a-x) (-a-x)^{\beta}\psi(-x) J_v(\omega x) \mathrm{d}x
\\&=\mathcal{O} \left(\max \left\{ \frac{1}{\omega^{\alpha+\delta+1}},\frac{\ln^{\mu}(\omega)}{\omega^{\min \left\{\beta+3/2,3/2 \right\}}} \right\}\right).
\endaligned
\end{equation*}
}


\noindent \textbf{Case (iv): $a< b<0$.} Using Equation~\eqref{eq17}, we get
\begin{equation*}
\aligned
& \int_{a}^b x^{\alpha}|x|^{\delta} \ln^{\mu}(x-a) (x-a)^{\beta}\psi(x) J_v(\omega |x|) \mathrm{d}x
\\&=(-1)^{\delta} \int_{a}^b x^{\alpha+\delta} \ln^{\mu}(x-a) (x-a)^{\beta}\psi(x) J_v(-\omega x) \mathrm{d}x
\\&=(-1)^{\alpha} \int_{-b}^{-a} x^{\alpha+\delta} \ln^{\mu}(-a-x) (-a-x)^{\beta}\psi(-x) J_v(\omega x) \mathrm{d}x
\\&=\mathcal{O} \left(\max \left\{\omega^{-\beta-3/2}\ln^{\mu}(\omega), \omega^{-3/2} \right\}\right).
\endaligned
\end{equation*}

The result in Equation \eqref{eq19} follows analogously by symmetry.
\end{proof}

With Lemma \ref{lemma3.2}, Lemma \ref{lemma3.4} and Lemma \ref{lemma3.5} at disposal, we now establish a connection between the Bessel function and the Laguerre or Hermite polynomial by using the Hilb-type formula in \eqref{eq9}. This connection enables us to derive the following asymptotic estimates presented in {Theorems} \ref{lemma3.6} and \ref{lemma3.7}.

\begin{theorem}\label{lemma3.6}
Let $\mu\in\mathbb{N}$, $\psi(x)\in C[a,b]$, $\psi'(x)\in L^1[a,b]$ and $\alpha,\beta,\tau$ are selected such that the following integrals are integrable. Then the following asymptotic results hold for $n\to\infty$.

\noindent(i) If $0=a<b<\infty$, we have
\begin{align}\label{eq20}
\int_0^{b} \ln^{\mu}(x) x^{\tau} (b-x)^{\beta}  {\rm e}^{-x} L_{n}^{(\alpha)}(x) \psi(x)\mathrm{d}x
=\mathcal{O}\left(\max \left\{{\ln^{\mu}(2\sqrt{n}) }{n^{\alpha-\tau-1}}
,{n^{\max \left\{\frac{\alpha-\beta}{2}-\frac{3}{4},\frac{\alpha}{2}-\frac{3}{4} \right\}}} \right\}\right),
\end{align}
and
\begin{align}\label{eq21}
\int_0^{b} \ln^{\mu}(b-x) x^{\tau} (b-x)^{\beta}  {\rm e}^{-x} L_{n}^{(\alpha)}(x) \psi(x)\mathrm{d}x
=
\mathcal{O}\left(\max \left\{{n^{\alpha-\tau-1}}
,{\ln^{\mu}(2\sqrt{n}) } {n^{\max \left\{\frac{\alpha-\beta}{2}-\frac{3}{4},\frac{\alpha}{2}-\frac{3}{4} \right\}}} \right\}\right).
\end{align}

\noindent(ii) If $0<a<b<\infty$, we have
\begin{align}\label{eq22}
\int_a^{b} \ln^{\mu}(x-a)  (x-a)^{\beta} \psi(x)  L_{n}^{(\alpha)}(x) \omega_{\alpha}(x)\mathrm{d}x
=
\mathcal{O}\left(\max \left\{{\ln^{\mu}(2\sqrt{n}) } n^{\frac{\alpha-\beta}{2}-\frac{3}{4}}
, {n^{ \frac{\alpha}{2}-\frac{3}{4} }} \right\}\right),
\end{align}
and
\begin{align}\label{eq23}
\int_a^{b} \ln^{\mu}(b-x)  (b-x)^{\beta} \psi(x)  L_{n}^{(\alpha)}(x) \omega_{\alpha}(x)\mathrm{d}x
=
\mathcal{O}\left(\max \left\{{\ln^{\mu}(2\sqrt{n}) } n^{\frac{\alpha-\beta}{2}-\frac{3}{4}}
, {n^{\frac{\alpha}{2}-\frac{3}{4} }} \right\}\right).
\end{align}
\end{theorem}

\begin{proof}
We begin by applying the substitution $x=t^2$ with $t>0$, it yields
\begin{align*}
\ln^{\mu}x=\ln^{\mu} t^2=2^{\mu} \ln^{\mu} t.
\end{align*}
By the Hilb-type formula \eqref{eq9}, the integral \eqref{eq20} becomes
\begin{align*}
\aligned
& \int_0^{b} \ln^{\mu}(x)  x^{\tau} (b-x)^{\beta} {\rm e}^{-x} L_{n}^{(\alpha)}(x) \psi(x) \mathrm{d}x
\\&=\frac{\Gamma(n+\alpha+1)}{n! \tilde{n}^{\frac{\alpha}{2}}}
\int_0^{b} \ln^{\mu}(x) x^{\frac{2\tau-\alpha}{2}}(b-x)^{\beta} J_{\alpha}  \big\{2(\tilde{n}x)^{1/2} \big\}
{\rm e}^{-\frac{x}{2}} \psi(x)  \mathrm{d}x +\mathcal{O} \left(n^{\frac{\alpha}{2}-\frac{3}{4}}\right)
\\&=\frac{2^{\mu+1}\Gamma(n+\alpha+1)}{n! \tilde{n}^{\frac{\alpha}{2}}}
\int_0^{\sqrt{b}} \ln^{\mu}(t) t^{2\tau-\alpha+1} \big(\sqrt{b}-t\big)^{\beta}  J_{\alpha}\big (2{\tilde{n}^{1/2}} t\big) \Psi(t) \mathrm{d}t +\mathcal{O} \left(n^{\frac{\alpha}{2}-\frac{3}{4}}\right),
\endaligned
\end{align*}
where $\tilde{n}=n+(\alpha+1)/2$ and $\Psi(t)={\rm e}^{-t^2/2}(\sqrt{b}+x)^{\beta}\psi(t^2)$. It is easy to verity that $\Psi(t)\in C \big[0,\sqrt{b}\big]$ and $\Psi(t)\in L^1 \big[0,\sqrt{b}\big]$. Therefore, by Lemma \ref{lemma3.2}, we have
\begin{align*}
\int_0^{\sqrt{b}} \ln^{\mu}(t) t^{2\tau-\alpha+1} \big(\sqrt{b}-t\big)^{\beta} J_{\alpha}\big (2{\tilde{n}^{1/2}} t\big) \Psi(t) \mathrm{d}t
=
\mathcal{O}\left(\max \left\{{\ln^{\mu}(2\sqrt{n}) }{n^{\alpha/2-\tau-1}}
,{n^{\max \left\{-\frac{\beta}{2}-\frac{3}{4},-\frac{3}{4} \right\}}} \right\}\right).
\end{align*}
In view of the asymptotic result \cite[(6.1.2)]{Abramowitz}
\begin{align}\label{eq24}
\sigma_n^{(\alpha)}=\frac{\Gamma(n+\alpha+1)}{n!}=\mathcal{O} \big(n^{\alpha}\big).
\end{align}
Then we have
\begin{align*}
\int_0^{b} \ln^{\mu}(x) x^{\tau} (b-x)^{\beta}  {\rm e}^{-x} L_{n}^{(\alpha)}(x) \psi(x)\mathrm{d}x
=
\mathcal{O}\left(\max \left\{{\ln^{\mu}(2\sqrt{n}) }{n^{\alpha-\tau-1}}
,{n^{\max \left\{\frac{\alpha-\beta}{2}-\frac{3}{4},\frac{\alpha}{2}-\frac{3}{4} \right\}}} \right\}\right).
\end{align*}

By employing similar techniques, along with Lemmas \ref{lemma3.2} and \ref{lemma3.4}, one can derive the remaining estimates \eqref{eq21}, \eqref{eq22} and \eqref{eq23}.

\end{proof}

\begin{theorem}\label{lemma3.7}
Let $\beta>-1$, $\mu\in\mathbb{N}$, $\psi(x)\in C[a,b]$ and $\psi'(x)\in L^1[a,b]$. Then, as $n\to\infty$, we have
\begin{align}\label{eq25}
\aligned
&\int_{a}^{b} \ln^{\mu}(b-x)  (b-x)^{\beta}  {\rm e}^{-x^2}  H_{n}(x) \psi(x)\mathrm{d}x
\\&=\left\{
\aligned
&\mathcal{O} \left(\ln^{\mu}(2\sqrt{n}) 2^n \Big(\frac n 2\Big) !\cdot
n^{\max \left\{-\frac \beta 2-1,-1\right\}} \right), &n \ \mathrm{is\ even}
\\&\mathcal{O} \left(\ln^{\mu}(2\sqrt{n}) 2^n \Big(\frac {n-1} 2\Big) !\cdot
n^{\max \left\{-\frac \beta 2-\frac 12,-\frac 12\right\}} \right), &n \ \mathrm{is\ odd}
\endaligned
\right.
\endaligned
\end{align}
and
\begin{align}\label{eq26}
\aligned
&\int_{a}^{b} \ln^{\mu}(x-a)  (x-a)^{\beta}  {\rm e}^{-x^2}  H_{n}(x) \psi(x)\mathrm{d}x
\\&=\left\{
\aligned
&\mathcal{O} \left(\ln^{\mu}(2\sqrt{n}) 2^n \Big(\frac n 2\Big) !\cdot
n^{\max \left\{-\frac \beta 2-1,-1\right\}} \right), &n \ \mathrm{is\ even}
\\&\mathcal{O} \left(\ln^{\mu}(2\sqrt{n}) 2^n \Big(\frac {n-1} 2\Big) !\cdot
n^{\max \left\{-\frac \beta 2-\frac 12,-\frac 12\right\}} \right), &n \ \mathrm{is\ odd}
\endaligned
\right. .
\endaligned
\end{align}
\end{theorem}

\begin{proof}
Using Equation \eqref{eq8}, the integral can be expressed as follows:
\begin{align*}
\aligned
 &\int_{a}^{b} \ln^{\mu}(b-x)  (b-x)^{\beta}  {\rm e}^{-x^2}  H_{n}(x) \psi(x)\mathrm{d}x
 \\&=
 \left\{
 \aligned
 &(-1)^{\frac n 2} 2^n \Big(\frac n 2\Big) !
 \int_{a}^{b} \ln^{\mu}(b-x)  (b-x)^{\beta}  {\rm e}^{-x^2}   L_{n/2}^{(-1/2)}(x^2)\psi(x) \mathrm{d}x,
 &n \textrm{ is even},
 \\&(-1)^{\frac {n-1}{2}} 2^n   \Big(\frac {n-1} {2}\Big) !
 \int_{a}^{b} \ln^{\mu}(b-x)  (b-x)^{\beta}  {\rm e}^{-x^2} x L_{(n-1)/2}^{(1/2)}(x^2)\psi(x) \mathrm{d}x,
 &n \textrm{ is odd}.
 \endaligned
  \right.
 \endaligned
\end{align*}
For even values of $n$, applying Hilb-type asymptotic formula \eqref{eq9} and Equation \eqref{eq19}, we obtain:
\begin{align*}
\aligned
 &\int_{a}^{b} \ln^{\mu}(b-x)  (b-x)^{\beta}   {\rm e}^{-x^2}   L_{n/2}^{(-1/2)}(x^2)\psi(x) \mathrm{d}x
 \\&=\mathcal{O}(n^{-\frac 1 4})\int_{a}^{b} |x|^{\frac 12} \ln^{\mu}(b-x)  (b-x)^{\beta}  {\rm e}^{-\frac{x^2}{2}} \psi(x) J_{-1/2} \big(2 \bar{n}^{1/2} |x| \big) \mathrm{d}x+\mathcal{O}(n^{-1})
 \\&=\mathcal{O} \left(\ln^{\mu}(2\sqrt{n}) n^{\max \left\{-\frac \beta 2-1,-1\right\}} \right),
 \endaligned
\end{align*}
where $\bar{n}=n/2+1/4$. For odd values of $n$, we similarly obtain:
\begin{align*}
\aligned
 &\int_{a}^{b} \ln^{\mu}(b-x)  (b-x)^{\beta} {\rm e}^{-x^2} x L_{(n-1)/2}^{(1/2)}(x^2)\psi(x) \mathrm{d}x
 \\&=\mathcal{O}(n^{\frac 1 4})\int_{a}^{b}  x|x|^{-\frac 12} \ln^{\mu}(b-x)  (b-x)^{\beta}  {\rm e}^{-\frac{x^2}{2}}\psi(x) J_{1/2} \big(2 \bar{n}^{1/2} |x| \big) \mathrm{d}x+\mathcal{O}(n^{-\frac 12})
  \\&=\mathcal{O} \left(\ln^{\mu}(2\sqrt{n}) n^{\max \left\{-\frac \beta 2-\frac 12,-\frac 12\right\}} \right).
 \endaligned
\end{align*}
Thus, the desired result in \eqref{eq25} is derived. A similar approach leads to the estimate in \eqref{eq26} by formula \eqref{eq9} and Equation \eqref{eq18}.
\end{proof}
\section{Asymptotics on Laguerre coefficients and convergence rates on Laguerre orthogonal projections for functions with algebraic and logarithmic regularities}

\subsection{Function with boundary regularities}

Consider the function
\begin{align}\label{eq27}
f(x)=x^{\delta}\ln^{\mu}(x)g(x),
\end{align}
where $\mu\in\mathbb{N}$ and $g(x)\in C^{\infty}[0,\infty)$, such that $\int_0^{\infty} {\rm e}^{-x} x^{\alpha} f(x) \mathrm{d}x<\infty$ with $\alpha>-1$.

\begin{theorem}\label{theorem4.1}
Let $f(x)$ be defined by \eqref{eq27}. Suppose that $f(x)$ satisfies the following conditions for $i=0,1,\ldots,k$ with $k$ being the smallest integer greater than or equal to $\alpha+2\delta+1/2${\rm :}
\begin{align}\label{eq28}
\lim_{x\to\infty} {\rm e}^{-x/2}x^{\alpha+i+1} f^{(i)}(x)=0,\
\int_b^{\infty} {\rm e}^{-x/2}x^{\alpha+k+1} \left|f^{(k+1)}(x)\right| \mathrm{d}x<\infty,
\end{align}
where $b>0$ is a fixed positive constant. Then, for $\alpha+\delta>-1$ and $n\to\infty$, the Laguerre coefficient $a_{n}(\alpha)$ given by \eqref{eq3} satisfies the following asymptotic estimate
\begin{align}\label{eq29}
| a_n (\alpha)|=\mathcal{O} \left(n^{ -\alpha-\delta-1} \ln^{\mu}(2\sqrt{n})\right).
\end{align}
\end{theorem}

\begin{proof}
Recalling the Rodrigues formula for the generalized Laguerre polynomials (see, e.g., \cite[(5.1.5)]{Szego})
\begin{align}\label{eq30}
{\rm e}^{-x} x^{\alpha}  L_{n}^{(\alpha)}(x)
=\frac{1}{(n)_{k}}\frac{\mathrm{d}^k}{\mathrm{d}x^k}  \Big\{{\rm e}^{-x} x^{\alpha+k}  L_{n-k}^{(\alpha+k)}(x) \Big\},
\end{align}
where $(n)_{k}=n(n-1)\cdots(n-k+1)$ denotes the falling Pochhammer symbol. Then, the Laguerre coefficients can be expressed as
\begin{align}\label{eq31}\begin{aligned}
 a_n (\alpha)=&\frac{1}{\sigma_n^{(\alpha)}n} \int_0^{\infty} f(x) \mathrm{d}\left\{{\rm e}^{-x} x^{\alpha+1}L_{n-1}^{(\alpha+1)}(x)\right\}\\
 =&\frac{1}{\sigma_n^{(\alpha)}n}\left({\rm e}^{-x} x^{\alpha+1}L_{n-1}^{(\alpha+1)}(x)f(x)\Big|_{0}^{\infty}-\int_{0}^{\infty}f^{\prime}(x){\rm e}^{-x} x^{\alpha+1}L_{n-1}^{(\alpha+1)}(x)\mathrm{d}x\right)\\
 =&\cdots\\
 =&\frac{(-1)^{k}}{\sigma_n^{(\alpha)}(n)_{k}}\int_0^{\infty} {\rm e}^{-x} x^{\alpha+k}  L_{n-k}^{(\alpha+k)}(x) f^{(k)}(x) \mathrm{d}x\\
 =&\frac{(-1)^{k}}{\sigma_n^{(\alpha)}(n)_{k}}\left(\int_{0}^{b}{\rm e}^{-x} x^{\alpha+k}  L_{n-k}^{(\alpha+k)}(x) f^{(k)}(x) \mathrm{d}x+\int_{b}^{\infty}{\rm e}^{-x} x^{\alpha+k}  L_{n-k}^{(\alpha+k)}(x) f^{(k)}(x) \mathrm{d}x\right),
\end{aligned}
\end{align}
where repeated integration by parts is applied, and  the maximum estimate of the Laguerre polynomials \eqref{eq10}, together with condition \eqref{eq28} ensures that
\begin{equation*}
  \lim_{x\to+\infty}{\rm e}^{-x} x^{\alpha+i+1}L_{n-i-1}^{(\alpha+i+1)}(x)f^{(i)}(x)=0,\quad i=0,1,\cdots,k.
\end{equation*}

{If $\delta>0$ is non-integer}, applying Leibniz rule and Fa{\`{a}} di Bruno's formula yields
\begin{align*}
\aligned
  f^{(k)}(x)&=\sum_{i=0}^{k}\binom{k}{i}(\delta)_{i}x^{\delta-i}\sum_{j=0}^{k-i}\binom{k-i}{j}
  \Big(\ln^{\mu}x\Big)^{(j)}g^{(k-i-j)}(x),
\endaligned
\end{align*}
which, for simplicity, can be written as
\begin{align*}
  f^{(k)}(x)=x^{\delta-k}\sum_{j=0}^{\min\{\mu,k\}}h_{\mu-j}(x)\ln^{\mu-j}x,
\end{align*}
where $h_{\mu-j}(x)\in C^{\infty}[0,\infty)$ are smooth functions corresponding to the derivatives of $g(x)$.

Using Equation \eqref{eq20}, the finite integral part of $a_{n}(\alpha)$ can be estimated by
\begin{align*}
\aligned
&\int_{0}^{b}{\rm e}^{-x} x^{\alpha+k}  L_{n-k}^{(\alpha+k)}(x)f^{(k)}(x)\mathrm{d}x
\\&=\sum_{j=0}^{\min\{\mu,k\}}\int_0^{b} {\rm e}^{-x} x^{\alpha+\delta}  L_{n-k}^{(\alpha+k)}(x)h_{\mu-j}(x)\ln^{\mu-j}x \mathrm{d}x
\\&=\sum_{j=0}^{\min\{\mu,k\}}\mathcal{O}\left(n^{k-\delta-1} \ln^{\mu-j}(2\sqrt{n})\right)
\\&= \mathcal{O} \left(n^{k-\delta-1} \ln^{\mu}(2\sqrt{n})\right),
\endaligned
\end{align*}
where the estimate follows from the assumption $k-\delta-1\geq\frac{\alpha+k}{2}-\frac{3}{4}$. For the second term in Equation \eqref{eq31}, integrating by parts once more gives
\begin{align}\label{eq34}
\aligned
  &\int_b^{\infty} {\rm e}^{-x} x^{\alpha+k}  L_{n-k}^{(\alpha+k)}(x) f^{(k)}(x) \mathrm{d}x
  \\&=-\frac{1}{n-k} {\rm e}^{-x} x^{\alpha+k+1}  L_{n-k-1}^{(\alpha+k+1)}(x) f^{(k)}(x) \big|_{x=b}
  \\&\quad-\frac{1}{n-k} \int_b^{\infty} {\rm e}^{-x} x^{\alpha+k+1}  L_{n-k-1}^{(\alpha+k+1)}(x) f^{(k+1)}(x) \mathrm{d}x.
  \endaligned
\end{align}
According to the estimate in Lemma \ref{lemma2.2}, the integral in \eqref{eq34} satisfies
\begin{align*}
\int_b^{\infty} {\rm e}^{-x} x^{\alpha+k+1}  L_{n-k-1}^{(\alpha+k+1)}(x) f^{(k+1)}(x) \mathrm{d}x
=\mathcal{O} \left(n^{\frac{\alpha+k}{2}+\frac{1}{4}} \right) .
\end{align*}
Combining this with the asymptotic result \eqref{eq24}, we conclude that
\begin{align*}
 a_n (\alpha)=\mathcal{O} \left(n^{ -\alpha-\delta-1} \ln^{\mu}(2\sqrt{n})\right).
\end{align*}

If $\delta$ is a nonnegative integer, we require an additional condition for $k$, namely $k>\max \{\delta,\mu \}$, then
\begin{align*}
f^{(k)}(x)=\sum_{i=0}^{k} x^{\delta-k+i} \sum_{j=0}^{\mu-1} \ln^{\mu}(x) \phi_{i,j}(x) +x^{\delta} \ln^{\mu}(x)\phi_0 (x),
\end{align*}
where $\phi_0 (x), \phi_{i,j}(x) \in C^{\infty}[0,+\infty)$. By the similar above argument, we get the desired result \eqref{eq29}.

If {$\delta<0$}, by a similar proof as above, without integrating by parts, we can also obtain the desired result \eqref{eq29}.
\end{proof}

\begin{example}\rm{}
Figure \ref{fig:graph1} illustrates the optimal decay of the Laguerre expansion coefficients $| a_n (\alpha)|$ {for} the function $f(x)=x^{\delta} \ln^{\mu}(x)$, with various values of $\alpha$, $\delta$, and $\mu$. As observed in the figure, the asymptotic orders are consistent with the results established in Theorem \ref{theorem4.1}.
\end{example}

\begin{figure}[H]
  \centering
  \includegraphics[scale=0.45]{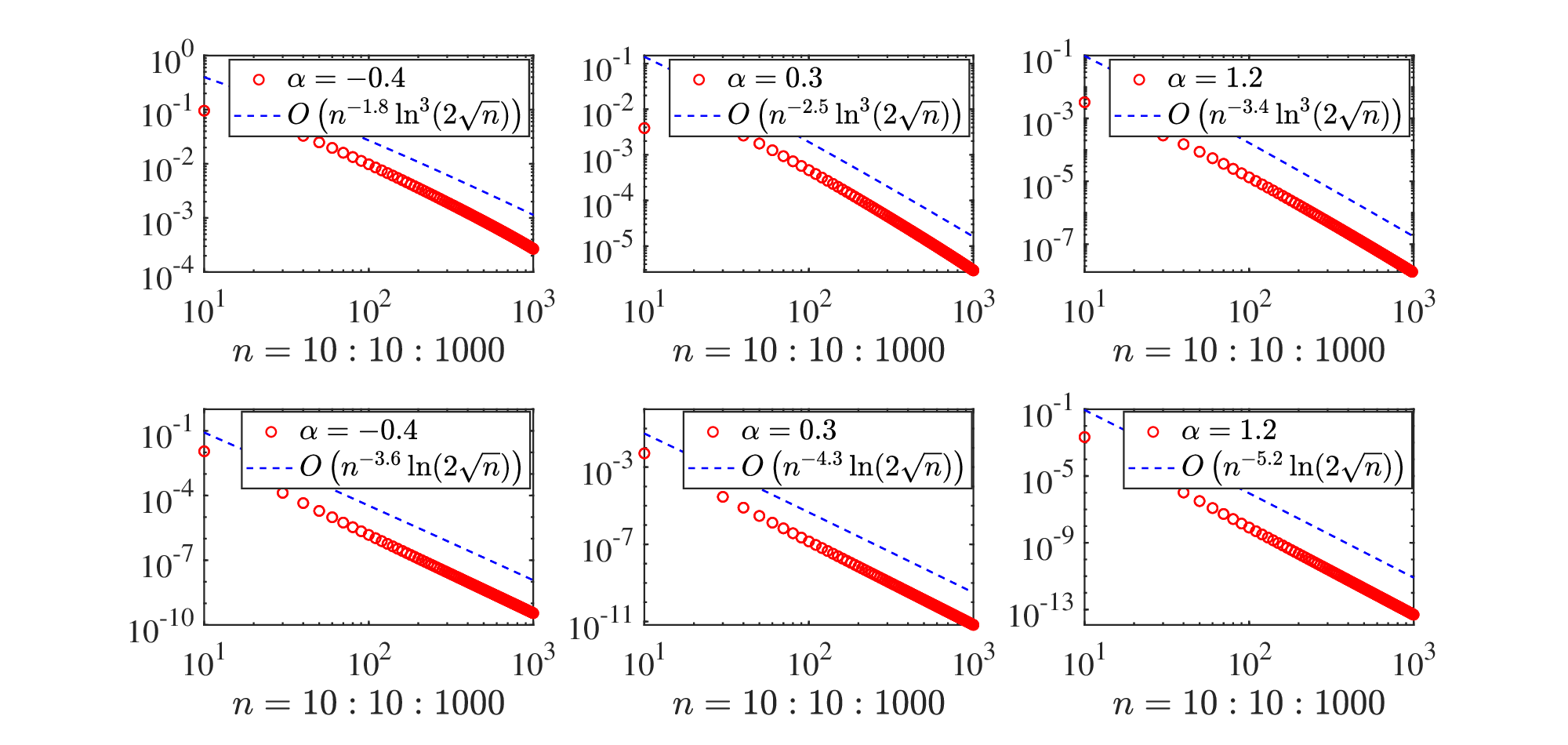}
  \caption{The asymptotic estimates of $| a_n (\alpha)|$ for $f(x)=x^{\delta} \ln^{\mu}(x)$: $\delta=1.2,\ \mu=3$ (first row); $\delta=3,\ \mu=3$ (second row). In all figures, the blue dashed lines represent the asymptotic orders provided in Theorem \ref{theorem4.1}.}
  \label{fig:graph1}
\end{figure}

\subsection{Function with interior regularities}

Consider the function
\begin{align}\label{eq35}
f(x)=|x-x_0|^{\gamma}\ln^{\mu}|x-x_0|g(x),\ |x_0|<\infty,
\end{align}
where $\mu\in\mathbb{N}$, $\gamma>0$, $g(x)\in C^{\infty}[0,\infty)$ such that $f(x)$ satisfies $\int_0^{\infty} {\rm e}^{-x} x^{\alpha} f(x) \mathrm{d}x<\infty$ for all $\alpha>-1$.

\begin{theorem}\label{theorem4.2}
Let $f(x)$ be defined by \eqref{eq35}. Suppose that $f(x)$ satisfies the following conditions for $i=0,1,\ldots,m$ with $k$ being the largest integer less than or equal to $\gamma$
\begin{align*}
\lim_{x\to\infty} {\rm e}^{-x/2}x^{\alpha+i+1} f^{(i)}(x)=0,\
\int_{b}^{\infty} {\rm e}^{-x/2}x^{\alpha+k+1} \left|f^{(k+1)}(x)\right| \mathrm{d}x<\infty,
\end{align*}
where $b>x_{0}$ is a fixed constant. Then, as $n\to\infty$, the Laguerre coefficient \eqref{eq3} satisfies
\begin{align}\label{eq36}
| a_n (\alpha)|=\mathcal{O} \left(n^{-\frac{\alpha+\gamma}{2}-\frac{3}{4}} \ln^{\mu}(2\sqrt{n})\right).
\end{align}
\end{theorem}
\begin{proof}
Observe that
\begin{align*}
a_n (\alpha)=\frac{1}{\sigma_n^{(\alpha)}} \left[\int_0^{x_0}+\int_{x_0}^{+\infty}\right] {\rm e}^{-x} x^{\alpha} f(x) L_{n}^{(\alpha)}(x) \mathrm{d}x,
\end{align*}
then we estimate the asymptotic behavior of each integral term separately.

\textbf{Case (i):} For $x\in[0,x_0]$, we have
\begin{align*}
f^{(k)}(x)=(x_0-x)^{\gamma-k}\sum_{j=0}^{\mu}\ln^{\mu-j}(x_{0}-x)e_{\mu-j}(x),
\end{align*}
where $e_{\mu-j}(x) \in C^{\infty}$ and vanish when $j>k$. Applying Rodrigues' formula \eqref{eq30}, we obtain
\begin{align}\label{eq38}
 I_1=\frac{1}{\sigma_n^{(\alpha)}} \int_0^{x_0} {\rm e}^{-x} x^{\alpha} f(x) L_{n}^{(\alpha)}(x) \mathrm{d}x=\frac{(-1)^{k}}
 { \sigma_n^{(\alpha)} (n)_{k}}\int_0^{x_0} {\rm e}^{-x} x^{\alpha+k}  L_{n-k}^{(\alpha+k)}(x) f^{(k)}(x) \mathrm{d}x.
\end{align}

If $\gamma$ is a positive integer, then $k=\gamma$. According to the assumptions, we have $\frac{\alpha+\gamma}{2}-\frac{3}{4}>-1$. By using Equation \eqref{eq21}, we get
\begin{align*}
\aligned
&\int_0^{x_0} {\rm e}^{-x} x^{\alpha+k}  L_{n-k}^{(\alpha+k)}(x)f^{(k)}(x)\mathrm{d}x
\\&=\sum_{j=0}^{\mu}\int_0^{x_0}{\rm e}^{-x} x^{\alpha+k}  L_{n-k}^{(\alpha+k)}(x) (x_0-x)^{\gamma-k}\ln^{\mu-j}(x_{0}-x)e_{\mu-j}(x) \mathrm{d}x
\\&=\sum_{j=0}^{\mu}\mathcal{O} \left(n^{\frac{\alpha+\gamma}{2}-\frac{3}{4}} \ln^{\mu-j}(2\sqrt{n}) \right)
\\&= \mathcal{O} \left(n^{\frac{\alpha+\gamma}{2}-\frac{3}{4}} \ln^{\mu}(2\sqrt{n}) \right).
\endaligned
\end{align*}
Combined with the normalization factor in \eqref{eq24}, this yields
\begin{align}\label{eq39}
I_1=\mathcal{O} \left(n^{-\frac{\alpha+\gamma}{2}-\frac{3}{4}} \ln^{\mu}(2\sqrt{n})\right).
\end{align}

If $\gamma$ is not an integer, then $k<\gamma<k+1$. By assumption, we have
\begin{align*}
  \frac{\alpha+2k-\gamma}{2}-\frac{3}{4}>\frac{\alpha+k}{2}-\frac{3}{4}>-1.
\end{align*}
Applying integration by parts to \eqref{eq38} yields
\begin{align*}
\aligned
 I_1&=\frac{(-1)^{k+1} }
 { \sigma_n^{(\alpha)} (n)_{k+1}}\int_0^{x_0} e^{-x} x^{\alpha+k+1}  L_{n-k-1}^{(\alpha+k+1)}(x) f^{(k+1)}(x) dx,
 \endaligned
\end{align*}
where
\begin{align*}
\aligned
f^{(k+1)}(x)&=(x_0-x)^{\gamma-k-1}\sum_{j=0}^{\mu}\ln^{\mu-j}(x_0-x)\tilde{e}_{\mu-j}(x),
\endaligned
\end{align*}
with $\tilde{e}_{\mu-j}(x) \in C^{\infty}[0,x_0]$. Using estimate \eqref{eq21} again, we obtain
\begin{align*}
\aligned
&\sum_{j=0}^{\mu}\int_0^{x_0} {\rm e}^{-x} x^{\alpha+k+1}  L_{n-k-1}^{(\alpha+k+1)}(x) (x_0-x)^{\gamma-k-1}\ln^{\mu-j}(x_0-x)\tilde{e}_{\mu-j}(x) \mathrm{d}x
\\&=\sum_{j=0}^{\mu}\mathcal{O}\left(n^{\frac{\alpha-\gamma}{2}+k+\frac{1}{4}} \ln^{\mu-j}(2\sqrt{n}) \right)
\\&= \mathcal{O} \left(n^{\frac{\alpha-\gamma}{2}+k+\frac{1}{4}} \ln^{\mu}(2\sqrt{n}) \right).
\endaligned
\end{align*}
{Together with \eqref{eq24}, it implies}
\begin{align}\label{eq41}
I_1=\mathcal{O} \left(n^{-\frac{\alpha+\gamma}{2}-\frac{3}{4}} \ln^{\mu}(2\sqrt{n})\right).
\end{align}

{Combining} \eqref{eq39} and \eqref{eq41}, we conclude that
\begin{align}\label{eq42}
\frac{1}{\sigma_n^{(\alpha)}} \int_0^{x_0} {\rm e}^{-x} x^{\alpha} f(x) L_{n}^{(\alpha)}(x) \mathrm{d}x
=\mathcal{O} \left(n^{-\frac{\alpha+\gamma}{2}-\frac{3}{4}} \ln^{\mu}(2\sqrt{n})\right).
\end{align}

\textbf{Case (ii):} For $x\in[x_0,+\infty)$, we write
\begin{align*}
\aligned
f^{(k)}(x)&=(x-x_0)^{\gamma-k}\sum_{j=0}^{\mu}\ln^{\mu-j}(x-x_0)p_{\mu-j}(x),
\endaligned
\end{align*}
where $p_{\mu-j} \in C^{\infty}[x_0,\infty)$ and vanish when $j>k$. By applying Rodrigues' formula again, we obtain
\begin{align*}
I_2=\frac{1}{\sigma_n^{(\alpha)}} \int_{x_0}^{\infty} {\rm e}^{-x} x^{\alpha} f(x) L_{n}^{(\alpha)}(x) \mathrm{d}x=\frac{(-1)^{k}}{\sigma_n^{(\alpha)} (n)_{k}} \int_{x_0}^{\infty} {\rm e}^{-x} x^{\alpha+k}  L_{n-k}^{(\alpha+k)}(x) f^{(k)}(x) \mathrm{d}x.
\end{align*}
The asymptotic estimate of $I_{2}$ follows by applying the same technique used in the proof of Theorem \ref{theorem4.1}.

Combining this with \eqref{eq42}, we conclude that the asymptotic behavior of $a_{n}(\alpha)$ satisfies
\begin{align*}
  a_{n}(\alpha)=\mathcal{O} \left(n^{-\frac{\alpha+\gamma}{2}-\frac{3}{4}} \ln^{\mu}(2\sqrt{n})\right),
\end{align*}
which is the desired result stated in \eqref{eq36}.
\end{proof}

\begin{remark}\rm{}
For $-1<\gamma \leq 0$ and $\alpha > -1$, a similar argument to that used in Theorem \ref{theorem4.2}, but without integration by parts, yields the following estimate
\begin{align*}
a_n (\alpha)=\mathcal{O} \left(n^{ -\frac{\alpha+\gamma}{2}-\frac{3}{4}} \ln^{\mu}(2\sqrt{n})\right).
\end{align*}
\end{remark}

\begin{example}\rm{}
Figure \ref{fig:graph2} illustrates the optimal decay of the Laguerre expansion coefficients $| a_n (\alpha)|$ of the function $f(x)=|x-0.3|^{\gamma}\ln^{\mu}|x-0.3|$ for different values of $\alpha$, $\gamma$, and $\mu$, respectively. As observed from the graph, the asymptotic behavior of the coefficients agrees with the estimates established in Theorem \ref{theorem4.2}.
\end{example}

\begin{figure}[H]
  \centering
  \includegraphics[scale=0.45]{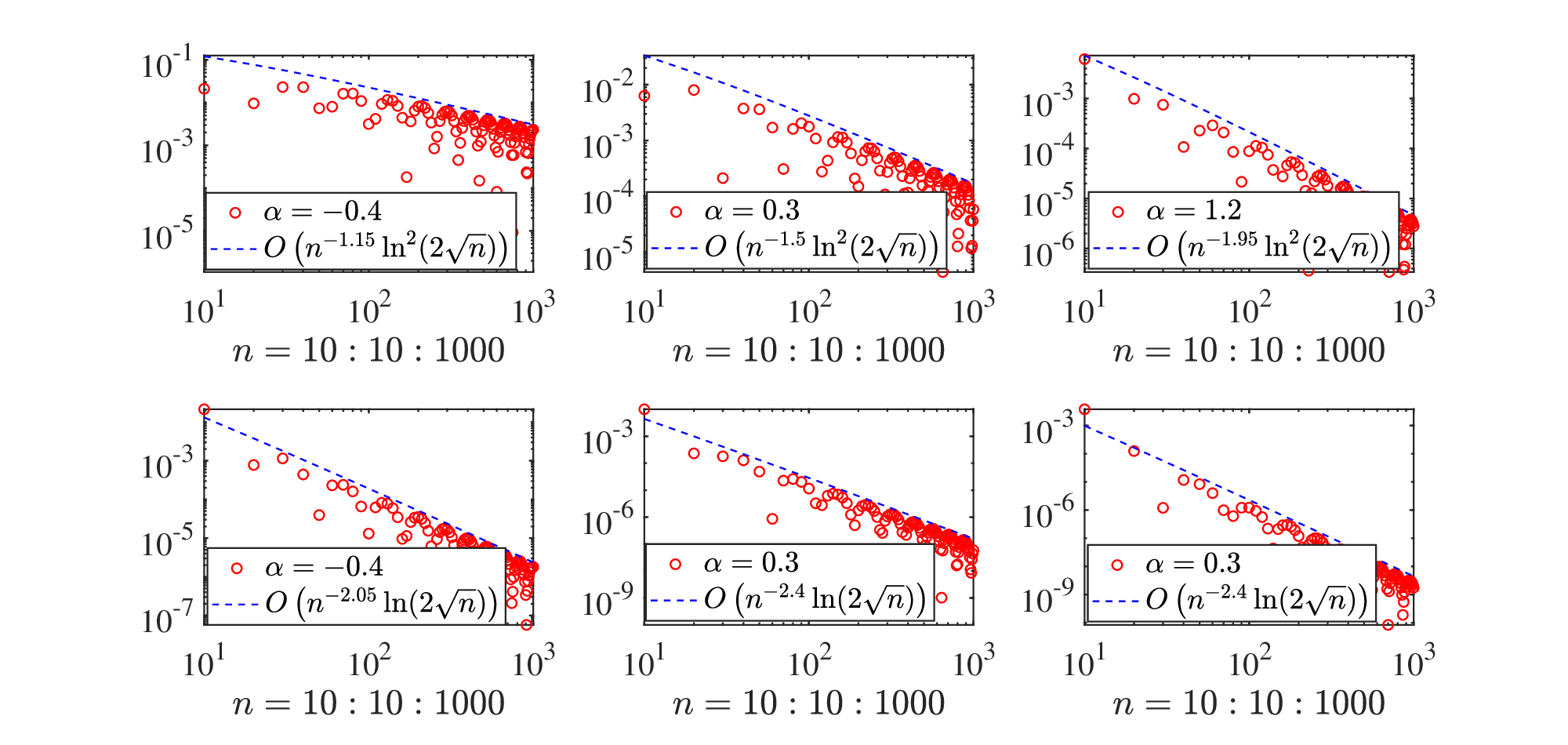}
  \caption{The asymptotic estimates of $| a_n (\alpha)|$ for $f(x)=|x-0.3|^{\gamma}\ln^{\mu}|x-0.3|$ with $\gamma=1.2,\ \mu=2$ (first row) and $\gamma=3,\ \mu=1$ (second row). In all figures, the blue dashed lines represent the asymptotic orders provided in Theorem \ref{theorem4.2}.}
  \label{fig:graph2}
\end{figure}

\subsection{The convergence rates on the Laguerre orthogonal projections}

For the function $f(x)=x^{\delta}\ln^{\mu}(x)g(x)$, it is straightforward to verify that $f\in L_{\omega_{\alpha}}^2 [0,+\infty)$ if $\alpha+2\delta>-1$. Similarly, for the function $f(x)=|x-x_0|^{\gamma}\ln^{\mu}|x-x_0|g(x)$, it can be observed that $f\in L_{\omega_{\alpha}}^2 [0,+\infty)$ if $\gamma>-1/2$. By applying the asymptotic estimates for the Laguerre coefficients of functions with algebraic and logarithmic singularities, as established in Theorems \ref{theorem4.1} and \ref{theorem4.2}, we will obtain the convergence rates of the Laguerre orthogonal projections for these respective functions.
\begin{theorem}\label{theorem4.3}
Let $f(x)$ be defined as in \eqref{eq27} and satisfy the assumptions in Theorem \ref{theorem4.1}.
Then for $\alpha+\delta>-1$, the Laguerre expansion follows as {$N\to\infty$} that
\begin{align*}
\aligned
&\|f(x)-S_N^{(\alpha)}[f](x)\|_{L_{\omega_{\alpha}}^2[0,+\infty)}=
\mathcal{O} \left(N^{ \frac{-\alpha-2\delta-1}{2}} \ln^{\mu}(2\sqrt{N})\right),\quad \alpha+2\delta>-1,
\\& \left\|{\rm e}^{-x/2} x^{\alpha/2} \big(f-S_{N}^{(\alpha)}[f] \big) \right\|_{L^{\infty}(0,\infty)}=\mathcal{O}\big(N^{-\frac{\alpha}{2}-\delta-\frac 14} \ln^{\mu}(2\sqrt{N})\big),\quad
\alpha+2\delta>-\frac 12.
\endaligned
\end{align*}
\end{theorem}
\begin{proof}
By Equation \eqref{eq4}, we have
\begin{align*}
\|f(x)-S_N^{(\alpha)}[f](x)\|_{L_{\omega_{\alpha}}^2[0,+\infty)}
=\left[{\sum_{n=N+1}^{\infty} a_n^2 (\alpha) \sigma_{n}^{(\alpha)}}\right]^{\frac 12},
\end{align*}
which, together with the estimate of $a_{n}(\alpha)$ in Theorem \ref{theorem4.1}.

Similarly, by substituting \eqref{eq4}, we get
\begin{align*}
{\rm e}^{-x/2} x^{\alpha/2} \left(f-S_{N}^{(\alpha)}[f] \right)
={\sum_{n=N+1}^{\infty} a_n (\alpha) \left({\rm e}^{-x/2} x^{\alpha/2} L_{n}^{(\alpha)}\right)},
\end{align*}
from which the second estimate follows directly by applying Lemma \ref{lemma2.2}, Theorem \ref{theorem4.1}, and integration by parts.
\end{proof}

\begin{example}\rm{}
Figure \ref{fig:graph3} shows the convergence of the Laguerre orthogonal projection $\|f-S_N^{(\alpha)}[f]\|_{\textrm{L}_{w_{\alpha}}^2 [0,+\infty)}$ for the function $f(x)=x^{\delta} \ln^{\mu}(x)$ with different values of $\alpha$, $\delta$, and $\mu$, respectively. As can be seen from the figure, these asymptotic orders are consistent with the results in Theorem \ref{theorem4.3}.
\end{example}

\begin{figure}[H]
  \centering
  \includegraphics[scale=0.45]{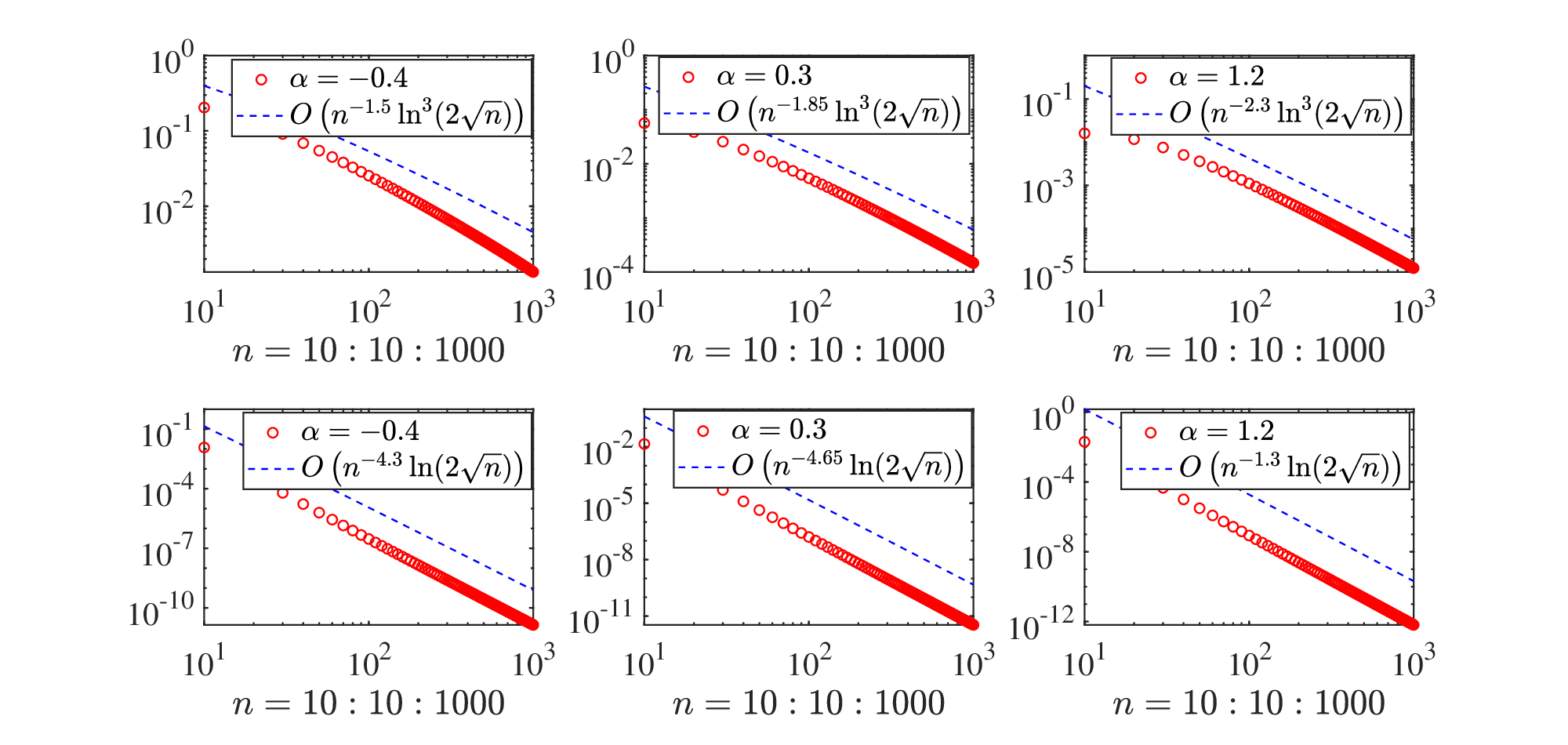}
  \caption{The asymptotic estimates of $\|f-S_N^{(\alpha)}[f]\|_{\textrm{L}_{w_{\alpha}}^2 [0,+\infty)}$ for $f(x)=x^{\delta} \ln^{\mu}(x)$: $\delta=1.2,\ \mu=3$ (first row); $\delta=4,\ \mu=1$ (second row). In all figures, the blue dashed lines represent the asymptotic orders provided in Theorem \ref{theorem4.3}.}
  \label{fig:graph3}
\end{figure}

\begin{theorem}\label{theorem4.4}
Let $f(x)$ be defined by \eqref{eq35} and satisfy the assumptions in Theorem \ref{theorem4.2}. Then, as {$N\to\infty$}, the Laguerre expansion follows that
\begin{align*}
\aligned
&\|f(x)-S_N^{(\alpha)}[f](x)\|_{L_{\omega_{\alpha}}^2[0,+\infty)}=
\mathcal{O} \left(N^{ -\frac{\gamma}{2}-\frac 1 4} \ln^{\mu}(2\sqrt{N})\right),\quad \gamma>-1/2,
\\& \left\|{\rm e}^{-x/2} x^{\alpha/2} \big(f-S_{N}^{(\alpha)}[f] \big) \right\|_{L^{\infty}(0,\infty)}=\mathcal{O}\Big(N^{-\frac{\gamma}{2}} \ln^{\mu}(2\sqrt{N})\Big),\quad
\gamma>0.
\endaligned
\end{align*}
\end{theorem}
\begin{proof}
By combining Theorem \ref{theorem4.2}, the results can be obtained using the same proof method as in Theorem \ref{theorem4.3}.
\end{proof}

\begin{example}\rm{}
Figure \ref{fig:graph4} shows the convergence of the Laguerre orthogonal projection $\|f-S_N^{(\alpha)}[f]\|_{\textrm{L}_{w_{\alpha}}^2 [0,+\infty)}$ for the function $f(x)=|x-0.3|^{\gamma}\ln^{\mu}|x-0.3|$ with different values of $\alpha$, $\gamma$, and $\mu$, respectively. As can be seen from the figure, these asymptotic orders are consistent with the results in Theorem \ref{theorem4.4}.
\end{example}

\begin{figure}[H]
  \centering
  \includegraphics[scale=0.45]{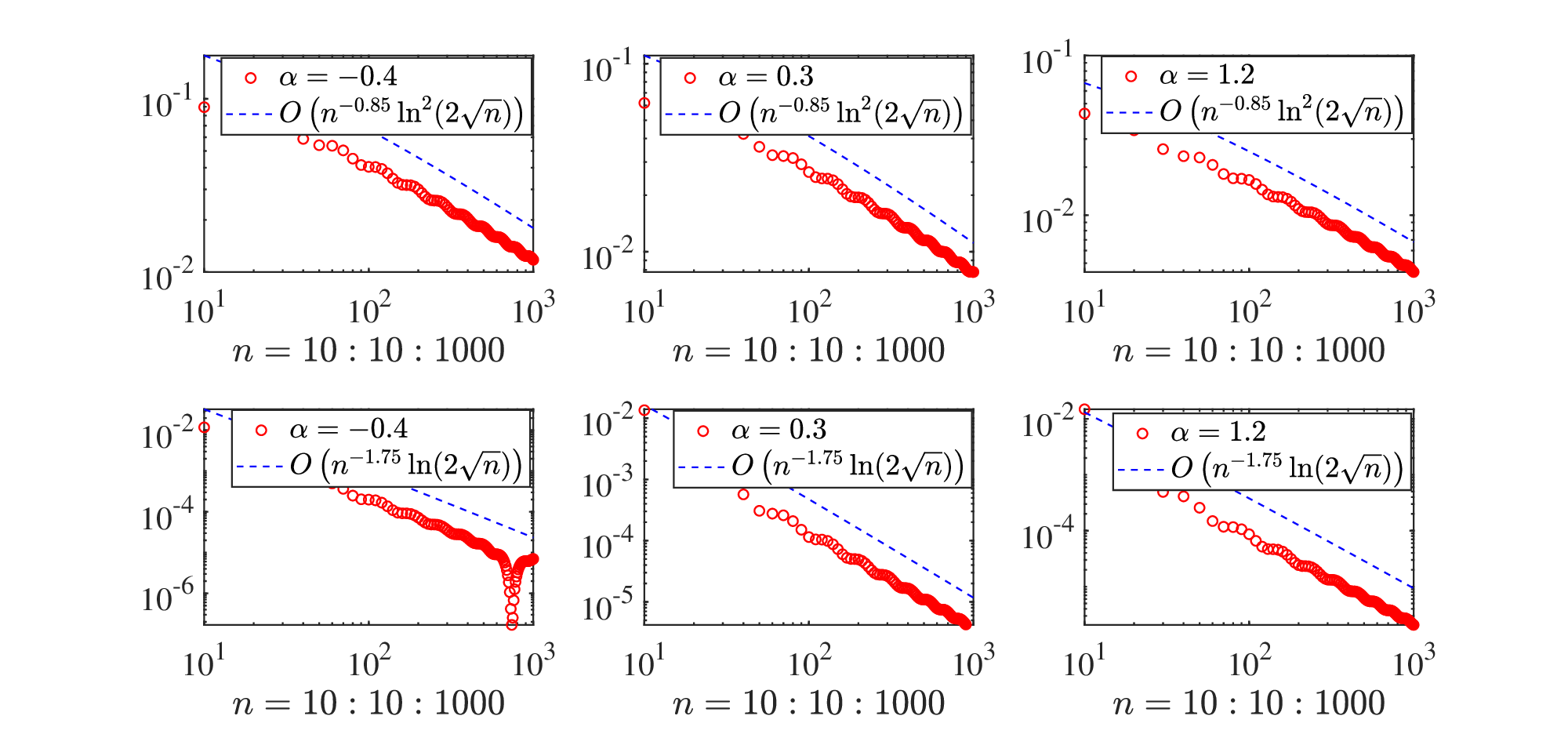}
  \caption{The asymptotic estimates of $\|f-S_N^{(\alpha)}[f]\|_{\textrm{L}_{w_{\alpha}}^2 [0,+\infty)}$ for $f(x)=|x-0.3|^{\gamma}\ln^{\mu}|x-0.3|$: $\gamma=1.2,\ \mu=2$ (first row); $\gamma=3,\ \mu=1$ (second row). In all figures, the blue dashed lines represent the asymptotic orders provided in Theorem \ref{theorem4.4}.}
  \label{fig:graph4}
\end{figure}

Moreover, {for} the non-uniformly Laguerre-weighted Sobolev space $H^{m,\alpha}(\Omega)$ \cite{Guo} with any integer $m\geq 0, \alpha>-1, \Omega=(0,+\infty)$, the weighted norm of $H^{m,\alpha}(\Omega)$ is defined by
\begin{align*}
\|u\|_{H^{m,\alpha}(\Omega)}=\left\{\sum_{q=0}^m \int_{0}^{+\infty}
{\rm e}^{-x}x^{\alpha+q} [u^{(q)}(x)]^2 \mathrm{d}x \right\}^{\frac{1}{2}},
\end{align*}
define
 \begin{align*}
f^{(q)}(x)=\sum_{n=0}^{\infty}a_n^{(q)}(\alpha+q) L_n^{(\alpha+q)}(x),\ q=0,1,...,m
\end{align*}
with
\begin{align*}
 a_n^{(q)} (\alpha+q)=\frac{1}{\sigma_n^{(\alpha+q)}} \int_0^{+\infty} {\rm e}^{-x} x^{\alpha+q} f^{(q)}(x) L_{n}^{(\alpha+q)}(x) \mathrm{d}x.
\end{align*}
From \eqref{eq1} and \eqref{eq31}, we get
\begin{align*}
\aligned
 a_n^{(q)} (\alpha+q)&=\frac{1}{\sigma_n^{(\alpha+q)}} \int_0^{+\infty} {\rm e}^{-x} x^{\alpha+q} f^{(q)}(x) L_{n}^{(\alpha+q)}(x) \mathrm{d}x
 \\&=\frac{ (-1)^{q} }{ \sigma_n^{(\alpha+q)} }
 (n+q)(n+q-1)\cdots(n+1)\int_0^{+\infty} {\rm e}^{-x} x^{\alpha} f(x) L_{n+q}^{(\alpha)}(x)  \mathrm{d}x
  \\&=\frac{(-1)^{q} \sigma_{n+q}^{(\alpha)} }{ \sigma_n^{(\alpha+q)} }
  (n+q)(n+q-1)\cdots(n+1)a_{n+q} (\alpha)
    \\&=a_{n+q} (\alpha) \mathcal{O}(1),
 \endaligned
\end{align*}
which together with Theorems \ref{theorem4.1} and \ref{theorem4.2}, we can obtain the following convergence rates.

For the function $f(x)=x^{\delta}\ln^{\mu}(x)g(x)$, it is easy to get $f\in H^{m,\alpha}(\Omega)$ if $\alpha+2\delta>m-1$. For the function $f(x)=|x-x_0|^{\gamma}\ln^{\mu}|x-x_0|g(x)$, it is realize that $f\in H^{m,\alpha}(\Omega)$ if $\gamma>m-\frac{1}{2}$.

\begin{corollary}
Let $f(x)$ be defined by \eqref{eq27} and satisfy the assumptions in Theorem \ref{theorem4.1}. Then for $\alpha+\delta>-1$ and {$N\gg 1$}, the Laguerre expansion follows that
\begin{align*}
\|f(x)-S_N^{(\alpha)}[f](x)\|_{H^{m,\alpha}(\Omega)}=
\mathcal{O} \left(N^{ \frac{m-\alpha-2\delta-1}{2}} \ln^{\mu}(2\sqrt{N})\right),\ \alpha+2\delta>m-1.
\end{align*}
\end{corollary}

\begin{corollary}
Let $f(x)$ be defined by \eqref{eq35} and satisfy the assumptions in Theorem \ref{theorem4.2}.
Then for {$N\gg 1$}, the Laguerre expansion follows that
\begin{align*}
\|f(x)-S_N^{(\alpha)}[f](x)\|_{H^{m,\alpha}(\Omega)}=
\mathcal{O} \left(N^{ \frac{m-\gamma}{2}-\frac 1 4} \ln^{\mu}(2\sqrt{N})\right),\ \gamma>m-\frac 12.
\end{align*}
\end{corollary}

\section{Asymptotics on Hermite coefficients and convergence rates on Hermite orthogonal projections for the functions with algebraic and logarithmic regularities}
\subsection{Function with interior regularities}

Consider the function
\begin{align}\label{eq44}
f(x)=|x-z_0|^{s}\ln^{\mu}|x-z_0|g(x),\ z_0\in(-\infty,+\infty),
\end{align}
where $\mu$ is a positive integer, $s>0$, $g(x)\in C^{\infty}(-\infty,+\infty)$ such that $\int_{-\infty}^{+\infty} {\rm e}^{-x^2} f(x) \mathrm{d}x<\infty$.

\begin{theorem}\label{theorem5.1}
Let $f(x)$ be defined by \eqref{eq44}. Suppose that $f(x)$ satisfies the following conditions for $i=0,1,\ldots,k$, with $k$ being largest integer less than or equal to $s$:
\begin{align*}
\aligned
&  \lim_{x\to-\infty} {\rm e}^{-\frac{x^2}{2}} f^{(i)}(x)=0,\
\lim_{x\to+\infty} {\rm e}^{-\frac{x^2}{2}} f^{(i)}(x)=0,
\\& \int_{-\infty}^{-b} {\rm e}^{-\frac{x^2}{2}} |f^{(k+1)}(x) |\mathrm{d}x<\infty,\
\int_{b}^{\infty} {\rm e}^{-\frac{x^2}{2}} |f^{(k+1)}(x) |\mathrm{d}x<\infty,
\endaligned
\end{align*}
where $b$ is a positive constant and $b>|z_0|$. Then, as $n\to\infty$, the Hermite coefficients in \eqref{eq6} satisfy
\begin{align}\label{eq45}
|h_n|=\mathcal{O}\left(n^{-\frac{n+s}{2}-1} \ln^{\mu}(2\sqrt{n})\right).
\end{align}
\end{theorem}
\begin{proof}
We begin by splitting $h_{n}$ as follows
\begin{align*}
h_n =\frac{1}{\gamma_n} \left[\int_{-\infty}^{z_0}+\int_{z_0}^{+\infty}\right] {\rm e}^{-x^2}  f(x) H_{n}(x) \mathrm{d}x=:I_{1}+I_{2}.
\end{align*}
We now estimate the asymptotic behavior of these two integrals. 

\textbf{Case (i):} For $x\in(-\infty,z_0]$, the $k$-th derivative of $f(x)$ can be written as
\begin{align*}
f^{(k)}(x)=(z_0-x)^{s-k}\sum_{j=0}^{\mu}\ln^{\mu-j}(z_0-x)q_{\mu-j}(x),
\end{align*}
where $q_{\mu-j}(x)\in C^{\infty}(-\infty,z_0]$. Recalling the Rodrigues' formula for Hermite polynomials
\begin{align*}
  \mathrm{e}^{-x^2}H_{n}(x)=(-1)^{k}\frac{\mathrm{d}^{k}}{\mathrm{d}x^{k}}\Big\{\mathrm{e}^{-x^2}H_{n-k}(x)\Big\},
\end{align*}
and performing repeated integration by parts, we obtain
\begin{align}\label{eq46}
\aligned
 I_1&=\frac{1}{\gamma_n} \int_{-\infty}^{z_0} f(x){\rm e}^{-x^2}  H_{n}(x) \mathrm{d}x
 \\&=\frac{1}{\gamma_n} \int_{-\infty}^{z_0} f^{\prime}(x){\rm e}^{-x^2}H_{n-1}(x)\mathrm{d}x
 \\&=\cdots
 \\&=\frac{1}{\gamma_n} \int_{-\infty}^{z_0} f^{(k)}(x){\rm e}^{-x^2}  H_{n-k}(x) \mathrm{d}x.
 \endaligned
\end{align}

If $s\in\mathbb{N}^{+}$, then $k=s$. The above integral can be split as
\begin{align}\label{eq47}
  \int_{-\infty}^{z_0} f^{(k)}(x){\rm e}^{-x^2}  H_{n-k}(x) \mathrm{d}x = \int_{-b}^{z_0} f^{(k)}(x)  {\rm e}^{-x^2}  H_{n-k}(x) \mathrm{d}x +\int_{-\infty}^{-b} f^{(k)}(x)  {\rm e}^{-x^2}  H_{n-k}(x)  \mathrm{d}x,
\end{align}
where $b>|z_{0}|$ is a constant. Substituting the expression for $f^{(k)}(x)$, the first integral in Equation \eqref{eq47} can be estimated as
\begin{align*}
\aligned
 &\int_{-b}^{z_0} f^{(k)}(x)  {\rm e}^{-x^2}  H_{n-k}(x) \mathrm{d}x
 \\&=\sum_{j=0}^{\mu}\int_{-b}^{z_0} \ln^{\mu-j}(z_0-x) {\rm e}^{-x^2}  H_{n-k}(x)q_{\mu-j}(x) \mathrm{d}x
 \\&=
 \left\{
 \aligned
 &\mathcal{O} \left(2^{n} \Big(\frac {n-k}{ 2}\Big) !n^{-1} \ln^{\mu}(2\sqrt{n})\right),
 &n-k\ \textrm{is even},
 \\&\mathcal{O} \left(2^{n} \Big(\frac {n-k-1}{ 2}\Big) !n^{-\frac 1 2} \ln^{\mu}(2\sqrt{n})\right),
 &n-k\ \textrm{is odd}.
 \endaligned
  \right.
 \endaligned
\end{align*}
For the second integral in \eqref{eq47}, we apply integration by parts and obtain
\begin{align}\label{eq48}
\aligned
&\int_{-\infty}^{-b} f^{(k)}(x)  {\rm e}^{-x^2}  H_{n-k}(x) \mathrm{d}x
\\&=-f^{(k)}(x) {\rm e}^{-x^2} H_{n-k-1}(x) \big |_{x=-b}
+\int_{-\infty}^{-b} f^{(k+1)}(x) {\rm e}^{-x^2} H_{n-k-1}(x) \mathrm{d}x.
\endaligned
\end{align}
According to Equation \eqref{eq11}, the two terms on the right-hand side of Equation \eqref{eq48} can be estimated as
\begin{align*}
  \Big|f^{(k)}(b) {\rm e}^{-b^2} H_{n-k-1}(b)\Big|=\mathcal{O}\left( \left(\frac {2}{{\rm e}} \right)^{\frac n2} n^{\frac{n-k}{2}-\frac{1}{2}} \right),
\end{align*}
and
\begin{align*}
 \int_{-\infty}^{-b} f^{(k+1)}(x)  {\rm e}^{-x^2}  H_{n-k-1}(x) \mathrm{d}x
=\mathcal{O}\left( \left(\frac {2}{{\rm e}} \right)^{\frac n2} n^{\frac{n-k}{2}-\frac{1}{2}} \right).
\end{align*}
Using Stirling's formula, and applying Equation \eqref{eq6}, we obtain
\begin{align*}
 I_1=\mathcal{O}\left( n^{-\frac{n+s}{2}-1} \ln^{\mu}(2\sqrt{n})\right).
\end{align*}

If $s>0$ is not an integer, we apply integration by parts again to $I_{1}$, yielding
\begin{align*}
 I_1=\frac{1}{\gamma_n} \int_{-\infty}^{z_0} f^{(k+1)}(x)  {\rm e}^{-x^2}  H_{n-k-1}(x) \mathrm{d}x.
\end{align*}
Here, $f^{(k+1)}(x)$ can be expressed as
\begin{align*}
f^{(k+1)}(x)=(z_0-x)^{s-k-1}\sum_{j=0}^{\mu}\ln^{\mu-j}(z_0-x)\tilde{q}_{\mu-j}(x),
\end{align*}
where $\tilde{q}_{\mu-j} \in C^{\infty}(-\infty,z_0]$. The integral can then be split again as
\begin{align}\label{eq49}
\aligned
 &\int_{-\infty}^{z_0} f^{(k+1)}(x)  {\rm e}^{-x^2}  H_{n-k-1}(x) \mathrm{d}x
\\&=\int_{-b}^{z_0} f^{(k+1)}(x)  {\rm e}^{-x^2}  H_{n-k-1}(x) \mathrm{d}x
+\int_{-\infty}^{-b} f^{(k+1)}(x)  {\rm e}^{-x^2}  H_{n-k-1}(x) \mathrm{d}x.
 \endaligned
\end{align}
Substituting the expression for $f^{(k+1)}(x)$, we obtain
\begin{align*}
\aligned
 &\int_{-b}^{z_0} f^{(k+1)}(x)  {\rm e}^{-x^2}  H_{n-k-1}(x) \mathrm{d}x
 \\&=\sum_{j=0}^{\mu}\int_{-b}^{z_0} (z_0-x)^{s-k-1} \ln^{\mu-j}(z_0-x) {\rm e}^{-x^2}  H_{n-k-1}(x)\tilde{q}_{\mu-j}(x) \mathrm{d}x
 \\&=
 \left\{
 \aligned
 &\mathcal{O}\left(2^{n} \Big(\frac {n-k-1}{ 2}\Big) !n^{\frac{k-s-1}{2}} \ln^{\mu}(2\sqrt{n})\right),
 &n-k-1\ \textrm{is even},
 \\&\mathcal{O}\left(2^{n} \Big(\frac {n-k-2}{ 2}\Big) !n^{\frac{k-s}{2}} \ln^{\mu}(2\sqrt{n})\right),
 &n-k-1\ \textrm{is odd}.
 \endaligned
  \right.
 \endaligned
\end{align*}
Similar to Equation \eqref{eq48}, the second integral in \eqref{eq49} can be estimated as
\begin{align*}
\int_{-\infty}^{-b} f^{(k+1)}(x)  {\rm e}^{-x^2}  H_{n-k-1}(x) \mathrm{d}x
=\mathcal{O}\left( \left(\frac {2}{{\rm e}}\right)^{\frac n 2} n^{\frac{n-k}{2}-1} \right).
\end{align*}
Using the Stirling formula and Equation \eqref{eq6}, we conclude that
\begin{align*}
 I_1=\mathcal{O}\left( n^{-\frac{n+s}{2}-1} \ln^{\mu}(2\sqrt{n})\right).
\end{align*}

\textbf{Case (ii):} For $x\in[z_0,+\infty)$, then
\begin{align*}
f^{(k)}(x)=(x-z_0)^{s-k}\sum_{j=0}^{\mu}\ln^{\mu-j}(x-z_0)r_{\mu-j}(x),
\end{align*}
where $r_{\mu-j}(x)\in C^{\infty}[z_0,\infty)$. By Rodrigues' formula, we get
\begin{align*}
 I_2=\frac{1}{\gamma_n} \int_{z_0}^{\infty} f^{(k)}(x) {\rm e}^{-x^2}  H_{n-k}(x) \mathrm{d}x.
\end{align*}
Following a similar approach to Case (i), we derive the desired result in Equation \eqref{eq45}.
\end{proof}

\begin{example}\rm{}
Figure \ref{fig:graph5} illustrates the decay rates of $\log_{n}| h_n|$ for the function $f(x)=|x-z_0|^{s}\ln^{\mu}|x-z_0|$ with various values of $s$ and $\mu$. As shown in the graph, the observed asymptotic behavior aligns with the predictions made in Theorem \ref{theorem5.1}.
\end{example}

\begin{figure}[H]
  \centering
  \includegraphics[scale=0.45]{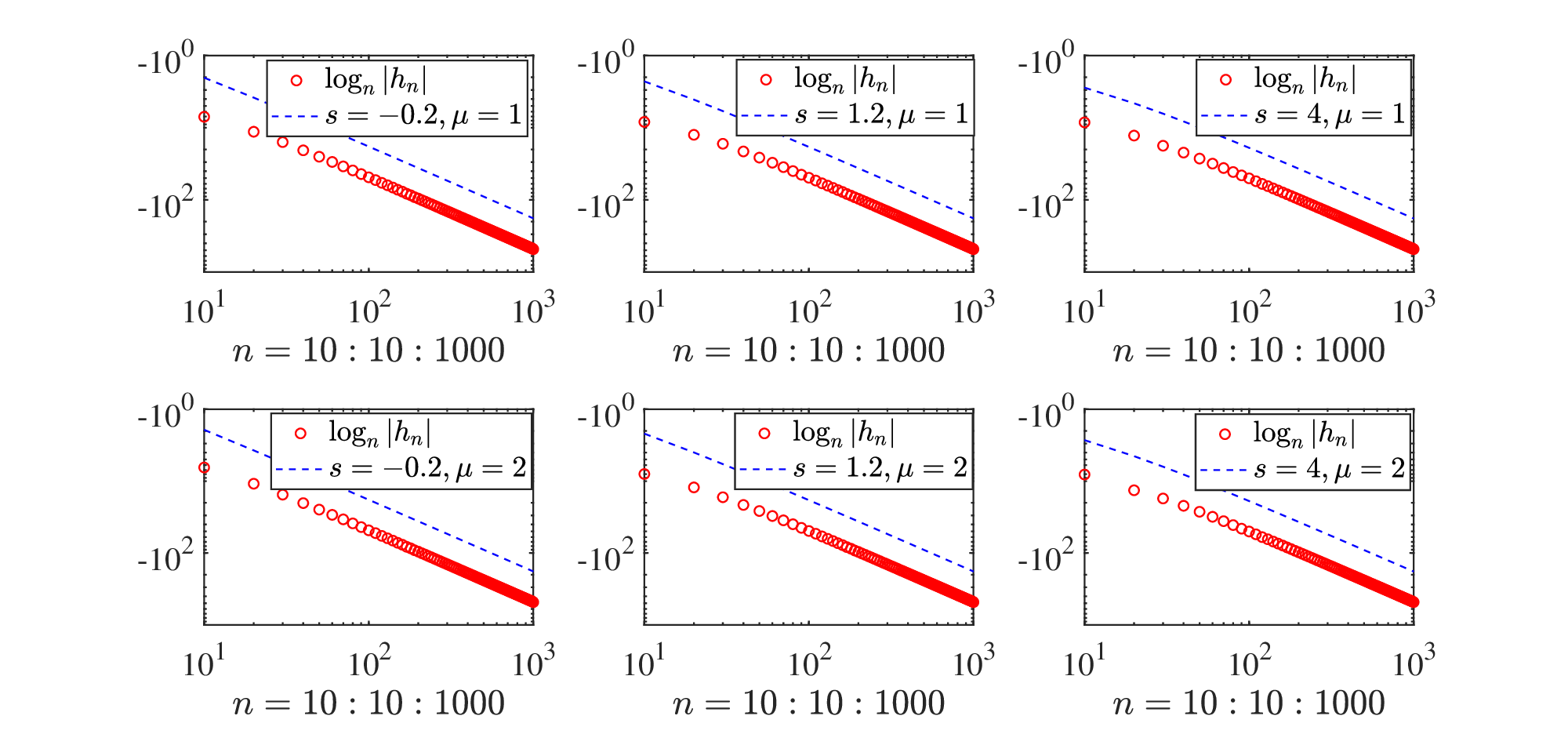}
  \caption{The asymptotic estimates of $\log_n | h_n|$ for $f(x)=|x-3|^{s}\ln^{\mu}|x-3|$. In all figures, the blue dashed lines represent the asymptotic orders provided in Theorem \ref{theorem5.1}.}
  \label{fig:graph5}
\end{figure}

\subsection{The convergence rates on the Hermite orthogonal projections}
The asymptotic behavior of the Hermite spectral expansion coefficients for functions with algebraic and logarithmic regularities at the interior, as described in Theorem \ref{theorem5.1}, enables us to determine the convergence rate of the Hermite orthogonal projection for a given function.

\begin{theorem}\label{theorem5.2}
Let $f(x)$ be defined by \eqref{eq44} and satisfies the assumptions given in Theorem \ref{theorem5.1}. Then, as {$N\gg 1$}, the Hermite expansion satisfies that
\begin{align*}
\aligned
&\|f(x)-S_N[f](x)\|_{L_{\omega}^2(R)}=
\mathcal{O} \left( N^{ -\frac{s}{2}-\frac 1 4} \ln^{\mu}(2\sqrt{N})\right).
\\& \left\|{\rm e}^{-x^2/2} \left(f-S_{N}[f] \right) \right\|_{L^{\infty}(R)}=\mathcal{O} \left( N^{ -\frac{s}{2}} \ln^{\mu}(2\sqrt{N})\right).
\endaligned
\end{align*}
\end{theorem}
\begin{proof}
By Equation \eqref{eq7}, we have
{
\begin{align*}
\|f(x)-S_N[f](x)\|_{L_{\omega}^2(R)}
&=\left[{\sum_{n=N+1}^{\infty} h_n^2  \gamma_{n}}\right]^{\frac 12},
\end{align*}}
which directly leads to the desired result by Theorem \ref{theorem5.1} and the method of integration by parts.

Similarly, from \eqref{eq7}, we obtain
\begin{align*}
{\rm e}^{-x^2/2}  \left(f-S_{N}[f] \right)
={\sum_{n=N+1}^{\infty} h_n  \left({\rm e}^{-x^2/2}  H_{n}\right)},
\end{align*}
which again leads to the desired result by applying Lemma \ref{lemma2.3}, Theorem \ref{theorem5.1}, and integration by parts.
\end{proof}

\begin{example}\rm{}
Figure \ref{fig:graph6} shows the convergence rates of the weighted Hermite truncation errors $\|f(x)-S_N[f](x)\|_{L_{\omega}^2(R)}$ for the function $f(x)=|x-z_0|^{s}\ln^{\mu}|x-z_0|$ with various values of $s$ and $\mu$. As can be seen from the graph, these asymptotic orders are consistent with the results of Theorem \ref{theorem5.2}.
\end{example}

\begin{figure}[H]
  \centering
  \includegraphics[scale=0.45]{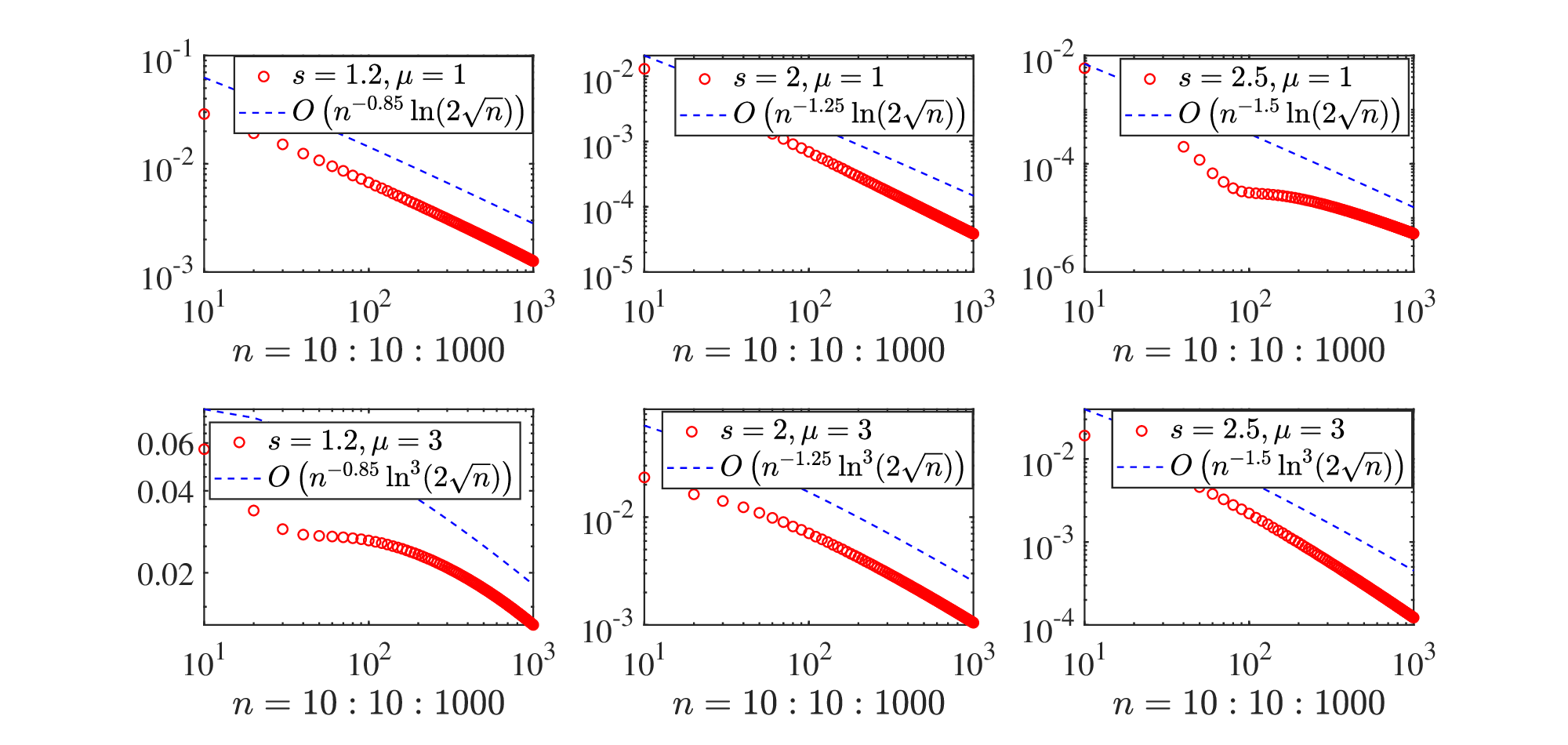}
  \caption{The asymptotic estimates of $\|f(x)-S_N[f](x)\|_{L_{\omega}^2(R)}$ for $f(x)=|x-3|^{s}\ln^{\mu}|x-3|$. In all figures, the blue dashed lines represent the asymptotic orders provided in Theorem \ref{theorem5.2}.}
  \label{fig:graph6}
\end{figure}

Moreover, {for} the Hermite-weighted Sobolev space $W^{m}(R)$ \cite{brezis2011functional} with any integer $m\geq 0$, the weighted norm of $W^{m}(R)$ is defined by
\begin{align*}
\|v\|_{W^{m}(R)}=\left\{\sum_{p=0}^m \int_{-\infty}^{+\infty}
{\rm e}^{-x^2} [v^{(p)}(x)]^2 \mathrm{d}x \right\}^{\frac{1}{2}},
\end{align*}
define
 \begin{align*}
f^{(p)}(x)=\sum_{n=0}^{\infty}h_n^{(p)} H_n(x),\ p=0,1,...,m
\end{align*}
with
\begin{align*}
 h_n^{(p)} =\frac{1}{\gamma_n} \int_{-\infty}^{+\infty}  f^{(p)}(x) H_{n}(x){\rm e}^{-x^2} \mathrm{d}x.
\end{align*}
From \eqref{eq6} and \eqref{eq46}, we get
\begin{align*}
\aligned
 h_n^{(p)} &=\frac{1}{\gamma_n} \int_{-\infty}^{+\infty}  f^{(p)}(x) {\rm e}^{-x^2} H_{n}(x) \mathrm{d}x
 \\&=\frac{1 }{ \gamma_n }
 \int_{-\infty}^{+\infty}  f(x) {\rm e}^{-x^2} H_{n+p}(x)  \mathrm{d}x
    \\&=h_{n+p} \mathcal{O}(1),
 \endaligned
\end{align*}
which together with Theorems \ref{theorem5.1} and \ref{theorem5.2}, we can obtain the following convergence rates.

\begin{corollary}
Let $f(x)$ be defined by \eqref{eq44} and satisfy the assumptions in Theorem \ref{theorem5.1}. Then for {$N\gg 1$}, the Hermite expansion follows that
\begin{align*}
\|f(x)-S_N[f](x)\|_{W^{m}(R)}=
\mathcal{O} \left(N^{ \frac{m-s}{2}-\frac 14} \ln^{\mu}(2\sqrt{N})\right).
\end{align*}
\end{corollary}

\section{Conclusions}

This paper investigates the optimal asymptotic behavior of the Laguerre and Hermite spectral expansion coefficients for functions with algebraic and logarithmic regularities. By applying the Hilb-type formula and deriving precise asymptotic estimates for integrals involving Laguerre and Hermite polynomials, we establish the optimal decay rates of the coefficients. These results are then used to characterize the decay of the corresponding Laguerre and Hermite spectral orthogonal projections. Numerical experiments are performed to validate the theoretical findings, demonstrating their accuracy and applicability. These results provide valuable insights into the spectral approximation of functions with regularities, paving the way for more efficient numerical methods in computational mathematics.


\vspace{2mm}

\section*{Conflict of Interest}

The authors declare that they have no conflict of interest.

\section*{Data Availability}

The code used in this work will be made available upon request to the authors.

\section*{References}


\end{document}